\documentclass[a4paper,11pt]{amsart}

\usepackage{amsfonts}
\usepackage{amssymb}
\usepackage{mathrsfs}
\usepackage{graphics}
\usepackage{tikz}
\usetikzlibrary{matrix,arrows}

\newcommand{\dl}{\lambda}

\newcommand{\mc}{\mathbb{C}}
\newcommand{\C}{\mathbb{C}}
\newcommand{\N}{\mathbb{N}}
\newcommand{\Z}{\mathbb{Z}}

\newcommand{\p}{\mathbb{P}}
\newcommand{\cald}{\mathcal{D}}

\newcommand{\tf}{\widetilde{\mathcal{D}}}
\newcommand{\fol}{\mathcal{F}}
\newcommand{\tilf}{\widetilde{\mathcal{F}}}

\newcommand{\partx}{{\partial /\partial x}}
\newcommand{\party}{{\partial /\partial y}}
\newcommand{\partz}{{\partial /\partial z}}

\newcommand{\R}{\mathbb{R}}

\newcommand{\psl}{\mathrm{PSL}\, (2, \mathbb{C})}

\newlength{\dhatheight}

\newcommand{\CP}{{\mathbb{C} \mathbb{P}}}

\usepackage[normalem]{ulem}
\usepackage{color}

\def\picill#1by#2(#3)#4
{\vbox to #2
	{\hrule width #1 height 0pt depth 0pt
		\vfill\special{illustration #3 scaled #4}}}

\newtheorem{theorem}{Theorem}[section]
\newtheorem{prop}[theorem]{Proposition}
\newtheorem{corol}[theorem]{Corollary}
\newtheorem{lemma}[theorem]{Lemma}
\newtheorem{obs}[theorem]{Remark}

\theoremstyle{definition}
\newtheorem{defi}[theorem]{Definition}

\theoremstyle{remark}
\newtheorem{remark}[theorem]{Remark}

\newtheorem{ex}{Example}

\usepackage{hyperref}

\begin{document}

\title[Global dynamics and singularities of foliations]{Global dynamics and perspectives on singularities of holomorphic foliations}

\author[J. Rebelo]{Julio Rebelo}
\address{Institut de Math\'ematiques de Toulouse ; UMR 5219, Universit\'e de Toulouse, 118 Route de Narbonne, F-31062 Toulouse, France.}
\email{\href{mailto:rebelo@math.univ-toulouse.fr}{rebelo@math.univ-toulouse.fr}}

\author[H. Reis]{Helena Reis}
\address{Centro de Matem\'atica da Universidade do Porto, Faculdade de Economia da Universidade do Porto, Portugal.}
\email{\href{mailto:hreis@fep.up.pt}{hreis@fep.up.pt}}



\begin{abstract}
In dimensions greater than or equal to~$3$, the local structure of a singular holomorphic foliation conceals a globally
defined foliation on the projective space of dimension one less. In this paper, we will
investigate how the global dynamics of the latter foliation exerts influence
on several problems that apparently have a purely local nature. In the course of the discussion,
a few recent results and open problems in the area will be reviewed as well.
\end{abstract}

\maketitle
\setcounter{tocdepth}{1}

\tableofcontents

\section{Introduction}

All {\it foliations}\, considered in this work are {\it holomorphic and (possibly) singular}. Whereas our main object are
$1$-dimensional holomorphic foliations and holomorphic/meromorphic vector fields, foliations of codimension~$1$ will also
play a role in the discussion especially when the ambient is of dimension~$3$, see Section~\ref{basics} for accurate definitions. Foliations of dimension~$1$
defined on some complex manifold $M$ will typically be denoted by $\fol$ while $\cald$ will stand for codimension~$1$
foliations. The purpose of this paper is to discuss recent results and open problems in the local theory of $1$-dimensional
foliations when the ambient manifold $M$ has dimension~$3$ though, occasionally, results and questions in higher dimensions
will also be included.

Foliations defined on complex surfaces, i.e. complex manifolds of dimension~$2$, are basically left aside in this paper
largely due to the fact that their local theory is much more advanced than their higher dimensional counterparts. Indeed, these
singularities are only mentioned in Section~\ref{basics} and, yet, with the simple purpose of identifying a few issues that
make them so special and amenable to very detailed analysis. In doing so, we will be able to single out one of the most
fundamental issues guiding the discussion conducted here: the presence of a global dynamical phenomenon intrinsically
attached to germs $\fol$ of $1$-dimensional foliations on $\C^n$ provided that $n\geq 3$. In slightly vague though more
incisive words, the understanding of a germ of $1$-dimensional foliation $\fol$ on $\C^n$, $n \geq 3$, passes through the
description of a foliation defined on $\CP^{n-1}$ which, as a global object, may exhibit a wild dynamical behavior (cf.
Section~\ref{core_dynamics}). The global foliation in question will usually be referred to as the {\it core foliation}\, of
the (local) foliation $\fol$. We will also use the phrase {\it core dynamics of $\fol$}\, to refer to the dynamics of the
core foliation associated with $\fol$.

The common thread of this paper is the existence and implications of a global dynamical system inherently attached to the structure
of a singularity of a $1$-dimensional, holomorphic foliation defined on $(\C^n,0)$ provided that $n \geq 3$. Basically,
we will discuss which types of results can be proved if the above mentioned dynamics can accurately be described and, similarly,
which general problems may provide us with the tools to ensure this dynamics is ``tame enough'' to be described, while
bearing in mind that in full generality this dynamics can be extremely wild.

The paper is structured as follows. In Section~\ref{basics}, we introduce standard terminology
and recall some basic features of singular foliations, in particular pointing out fundamental issues setting apart
foliations of dimension~$1$ and foliations of codimension~$1$.
Then we focus on the special case of singularities
of foliations defined on $(\C^2,0)$. Whereas this case is clearly distinguished by the fact that the foliations
are simultaneously of dimension~$1$ and of codimension~$1$, we discuss to a rather non-trivial extent the main $2$-dimensional
issues allowing for the existence of such a sophisticated theory covering truly fine issues.

In Section~\ref{core_dynamics}, we introduce a fundamental object that exists for singularities of ($1$-dimensional)
foliations defined on $(\C^n,0)$ provided that $n \geq 3$, namely: the {\it core dynamics}. This is a global
foliation defined on $\CP^{n-1}$ whose (global) dynamics poses a fundamental obstacle towards
the full understanding of the initial singular point. In particular, we explain how this core dynamics plays a
major role in problems about existence of separatrices for codimension~$1$ foliations on $(\C^3,0)$. Also,
we show how its very existence basically rules out any hope of achieving a full understanding of large classes
of singular points.

The remainder of this survey is devoted to more advanced material, in particular touching on quite a few open
problems of current interest. Section~\ref{resolution_blowups} contains a detailed review of resolution theorems
for singular points of $1$-dimensional foliations on $(\C^3,0)$. The first definitive resolution theorem in this
context was obtained by McQuillan and Panazzolo in \cite{danielMcquillan} which, in turn, relies heavily on
Panazzolo's algorithm introduced in \cite{daniel2}. More recently, a different proof based on Zariski general
point of view was obtained in \cite{helenajulio_RMS} which can be seen as the completion of the previous work carried out
by Cano-Roche-Spivakovsky \cite{C-R-S}. Despite the undisputed importance of resolution theorem,
it seems these results are still not as widely known as one would expect and, for
this reason, we thought useful to conduct
a thorough discussion about the content of the resolution theorems in \cite{danielMcquillan} and in \cite{helenajulio_RMS},
highlighting virtues and potential limitations.

In Section~\ref{Allsortsofseparatrices}, we consider the fundamental problem of existence of separatrices that, roughly
speaking, concerns the existence of germs of analytic sets invariant by (germs of) singular foliations. The discussion
is essentially conducted in $(\C^3,0)$. Expanding on the discussion of Section~\ref{core_dynamics}, we consider
the existence of separatrices for codimension~$1$ foliation spanned by two commuting vector fields and state
Theorem~\ref{teo_RR_separatrix} affirmatively answering this question. We also detail the general strategy for proving this theorem
which, in turn, emphasizes a few often overlooked points in resolution theorems for foliations as well as
the major role played by the general question of ``taming a core dynamics''. The second part of this section,
we review some results on the existence of separatrices for foliations of dimension~$1$. The nature of this second
problem is far more topological/geometric and ``core dynamics'' plays a much smaller role. Yet, some of the results
will find applications in the last section.

Finally, in Section~\ref{semicompletevectorfields} we discuss a particular class of singularities of foliations
of dimension~$1$, namely those supporting {\it semicomplete vector fields}. Albeit small in an appropriate sense,
this class of singularities
has rather distinguished properties and quite a few applications that make it worth studying. The section will precisely
begin with proper definitions and a general discussion of applications. Once the basic setting is in place, we will go on
to discuss two fundamental problems on the area: the first problem can vaguely be stated by asking {\it how wild
can the core dynamics be in this class of foliations}? The main results here stem from the seminal paper
of A. Guillot about Halphen vector fields and their role in ${\rm SL}\, (2,\C)$-actions, see \cite{guillot-IHES}.
The second problem aims at quantifying how ``degenerate'' a singularity in this class can be. This second problem
stems from a well known question raised long ago by E. Ghys and the topic has applications in the study of automorphism
groups of compact complex manifolds.

\medskip
\noindent
\textbf{Acknowledgment.} We are grateful to many of our colleagues for several discussion over the years.
A special thanks is due to D. Panazzolo who has helped us to understand many subtle points of his fundamental work
on desingularization theorems. We also thank J.-F. Mattei for many discussions and explanations about the vast and
fundamental work in singularity theory he has accomplished with his collaborators.

J. Rebelo and H. Reis were partially supported by CIMI through the project ``Complex dynamics of group actions, Halphen
and Painlev\'e systems''. H. Reis was also partially supported under the project ``Means and Extremes in Dynamical Systems''
with reference PTDC/MAT-PUR/4048/2021 and also by CMUP, member of LASI, which is financed by national funds through FCT
Funda\c{c}\~ao para a Ci\^encia e Tecnologia, I.P., under the project UIDB/00144/2020 as also supported .


\section{Basics in the local theory of foliations and the special case of dimension~$2$}\label{basics}

It is convenient to begin by recalling the definition of $1$-dimensional singular holomorphic foliation. First, let
$X = f_1 \partial /\partial x_1 + \cdots + f_n \partial /\partial x_n$ be a non-trivial holomorphic vector field defined
on an open set $V$ of $\C^n$. The singular set ${\rm Sing}\, (X)$ of $X$ is then given by $\bigcap_{i=1}^n \{ f_i =0\}$.
It is a (proper) analytic subset of $V$ and it is well known that ${\rm Sing}\, (X)$ has codimension~$1$ if and only if
the coordinate functions $f_i$ admit a non-trivial common factor.

We are then able to define singular holomorphic foliations as they will be considered throughout this work. Let $M$ be a
complex manifold and consider a covering $\{ (U_k, \varphi_k)\}$ of $M$ by coordinate charts. We denote by $n$ the dimension
of $M$ so that $\varphi_k (U_k)$ is an open set of $\C^n$.

\begin{defi}
	\label{1-dimensional_foliation_definition}
	Let $M$ and $\{ (U_k, \varphi_k)\}$ be as above. A singular holomorphic foliation $\fol$ of dimension~$1$ on $M$ consists of a collection
	of holomorphic vector fields $Y_k$ satisfying the following conditions:
	\begin{itemize}
		\item For every $k$, $Y_k$ is a holomorphic vector field defined on $\varphi_k (U_k) \subset \C^n$ whose
		singular set has codimension at least~$2$.
		
		\item Whenever $U_{k_1} \cap U_{k_2} \neq \emptyset$, we have
		$Y_{k_1} = g_{k_1 k_2} . (\varphi_{k_2} \circ \varphi_{k_1}^{-1})^{\ast} Y_{k_2}$ for
		some nowhere vanishing holomorphic function $g_{k_1 k_2}$.
	\end{itemize}
\end{defi}

The {\it singular set}\, ${\rm Sing}\, (\fol)$ of a foliation $\fol$ is then defined as the union over~$k$ of the sets $\varphi_k^{-1}
({\rm Sing}\, (Y_k)) \subset M$. Therefore the singular set of any holomorphic foliation has codimension at least two. In particular,
a foliation {\it has no divisor of zeros}.

Conversely, we say that a holomorphic vector field $Y$ is a {\it local representative}\, of the $1$-dimensional foliation $\fol$ if $Y$
is tangent to $\fol$ and the singular set of $Y$ has codimension at least~$2$. It is clear that representative vector fields are locally
unique up to multiplication by an invertible holomorphic function.

Analogously, we might also consider a differential $1$-form $\omega$ on $V \subseteq \C^n$, $\omega = g_1 dx_1 + \cdots + g_n dx_n$.
Again the singular set ${\rm Sing}\, (\omega)$ of $\omega$ is given by the intersection $\bigcap_{i=1}^n \{ g_i =0\}$ and it is an
analytic set of $V$ which has codimension~$1$ if and only if there is a non-trivial common factor for the functions $g_1 , \ldots ,
g_n$. Away from its singular points, the kernel of $\omega$ defines a distribution of complex hyperplanes on $V$. If in
Definition~\ref{1-dimensional_foliation_definition} we replace ``local vector fields'' by ``integrable local $1$-forms'',
we obtain the notion codimension~$1$ foliations. More precisely, we have:

\begin{defi}
	\label{codimension-1_foliation_definition}
	Let $M$ be a complex manifold and $\{ (U_k, \varphi_k)\}$ a covering of $M$ by coordinate charts.
A singular holomorphic foliation $\cald$ of codimension~$1$ on $M$ consists of a collection
of differential $1$-forms $\Omega_k$ satisfying the following conditions:
	\begin{itemize}
		\item For every $k$, $\Omega_k$ is a differential $1$-form defined on $\varphi_k (U_k) \subset \C^n$ with
		singular set of codimension at least~$2$ and such that $\Omega_k \wedge d\Omega_k$ vanishes identically.
		
		\item Whenever $U_{k_1} \cap U_{k_2} \neq \emptyset$, we have
		$\Omega_{k_1} = g_{k_1 k_2} . (\varphi_{k_2} \circ \varphi_{k_1}^{-1})^{\ast} \Omega_{k_2}$
		for some nowhere vanishing holomorphic function $g_{k_1 k_2}$.
	\end{itemize}
\end{defi}

In particular the singular set ${\rm Sing}\, (\cald)$ of $\cald$ again has codimension at least~$2$. The notion of {\it representative
$1$-form} for a codimension~$1$ foliation $\cald$ is defined analogously to the notion of representative vector fields in the case
of foliations with dimension~$1$.

Whereas our main focus here is on germs of foliations, or in slightly more concrete terms, on foliations defined on a neighborhood
of the origin in $\C^n$, the reader will notice that the global point of view considered in Definitions~\ref{1-dimensional_foliation_definition}
and~\ref{codimension-1_foliation_definition} is really indispensable to investigate the local structure of the singular point. Indeed,
globally defined foliations - and in particular the ``global dynamical phenomenon'' mentioned in the Introduction - will come to fore
in the context of birational theory of foliations which is needed, for example, if one seeks to ``resolve singular points''.

It is also convenient to complement the above definitions with a couple of comments.

\begin{obs}
{\rm Already on $(\C^n,0)$, $n \geq 3$, it is easy to see a first fundamental difference between foliations of dimension~$1$
and foliations of codimension~$1$ arising from the Frobenius condition. To formulate this, note that any choice of local
holomorphic functions $f_1, \ldots , f_n$ naturally gives rise to two (singular) distributions: one of lines and one of
hyperplanes. In fact, to the collection of functions $f_1, \ldots , f_n$, we may associate the vector field $Y = f_1
\partial /\partial x_1 + \cdots +f_n \partial /\partial x_n$ or the $1$-form $\Omega = f_1 dx_1 + \cdots + f_n dx_n$.
Whereas the local integral curves of $Y$ always yield a $1$-dimensional foliation $\fol$, the Frobenius condition for
$\Omega$ to yield a codimension~$1$ foliation is non-trivial and amounts to requiring the $3$-form $\Omega \wedge d\Omega$
to vanish identically which, in turn, leads to a highly non-trivial PDE system involving the functions $f_1, \ldots , f_n$.}
\end{obs}

\begin{obs}
{\rm In general, foliations of dimension~$1$ are very abundant, at least in algebraic manifolds, and they may have an
extremely complicated dynamical behavior, more on this in Section~\ref{core_dynamics}, see also \cite{classicalilyashenko},
\cite{lorayandjulio}. This contrasts with the case of codimension~$1$ foliations that are far more rigid and in several cases
amenable to classification, at least at conjectural level, all codimension~$1$ foliations on, say, $\CP^n$ should
be transversely homogeneous or can be obtained as a suitable pull-back of a foliaion defined on a surface. For an interesting
discussion of several of global aspects of codimension~$1$ foliations, we refer the reader to \cite{touzet}.}
\end{obs}

A basic object in the local theory of foliations that has largely motivated its early development is the notion of
{\it separatrix}. Although the definition of {\it separatrix}\, depends on the dimension of the foliation, the cases
of foliations of dimension~$1$ and of codimension~$1$ can naturally be formulated together.

\begin{defi}\label{definition_separatrix}
Let $\fol$ (resp. $\cald$) be a foliation of dimension~$1$ (resp. codimension~$1$) on $(\C^n , 0)$. A separatrix $S$ for $\fol$
(resp. $\cald$) is the germ of an irreducible analytic set of dimension~$1$ (resp. codimension~$1$) containing $0 \in \C^n$ and
invariant by $\fol$ (resp. $\cald$).
\end{defi}

Separatrices are objects of natural interest since they fit the framework of ``invariant manifolds'' in dynamical systems. In
particular, their presence may enable one to reduce the dimension of the phase space of the system in question. Yet, in the local
theory of foliations as discussed here, the notion of separatrix first appeared in the classical work of Briot and Bouquet
\cite{briotbouquet} where it was claimed that every foliation on $(\C^2,0)$ admits at least one separatrix. Much later, R.
Thom has sought to generalize the existence of separatrices for every codimension~$1$ foliation on $(\C^3,0)$. The first
example of a codimension~$1$ foliation on $(\C^3,0)$ without separatrix was, however, found by J.-P. Jouanolou \cite{jouanolou}.
As it will be seen in the next section, Jouanolou's counterexample relies on the core dynamics of certain $1$-dimensional foliations
on $(\C^3,0)$. Let us also point out that a gap in the arguments of Briot and Bouquet was later found and the existence
of separatrices for foliations on $(\C^2,0)$ was firmly established by Camacho and Sad in \cite{camachosad}.

Another fundamental notion in the theory of singularities of foliations is the notion of {\it eigenvalues of a foliation at
a singular point}. To abridge the discussion, for the time being let us restrict ourselves to the case of {\it foliations of
dimension~$1$} (see Section~\ref{resolution_blowups} for a more general discussion). Without loss of generality, it suffices to
consider the case of foliations $\fol$ of dimension~$1$ defined on $(\C^n,0)$. Given $\fol$ as above, up to reducing the neighborhood
of the origin, there is a holomorphic vector field $Y$ whose zero-set has codimension at least~$2$ and such that $\fol$ is nothing but
the foliation induced by the local orbits of $Y$. As mentioned, $Y$ is said to be a representative of $\fol$ and, while $Y$ is not
unique, two representative vector fields for the same foliation $\fol$ differ by multiplication by an invertible holomorphic function.

\begin{defi}
Let $\fol$ be a $1$-dimensional holomorphic foliation on $(\C^n ,0)$ and let $Y$ denote a representative vector field for $\fol$.
Assume that $\fol$ is singular at the origin, i.e., $Y(0) =0$. Then
the eigenvalues of $\fol$ at $0 \in \C^n$ are the eigenvalues of the Jacobian matrix $D_0Y$.
\end{defi}

Since $Y$ is well defined up to
multiplication by an invertible holomorphic function, there follows that the eigenvalues of $\fol$ at $0 \in \C^n$ are well defined
only up to simultaneous multiplication by a non-zero constant.

\subsection{Blow-ups and dicritical singularities} Blow-ups are a standard tool to produce non-trivial birational maps and to
understand the local structures of singular points, whether these are ``singularities of the ambient space'' or ``singularities
of a foliation on a smooth space''. The transform of a foliation under a blow-up map is called the {\it blow-up}\, of the foliation.
The blown-up space, however, contains an exceptional divisor which may or may not be invariant by the transformed foliation. This
issue gives rise to the notion of {\it dicritical foliation at a singular point}.

\begin{defi}
	\label{dicritical_foliation}
Let $M$ be a complex manifold equipped with a holomorphic foliation $\mathcal{H}$ and consider a blow-up map $\pi \, : \widetilde{M}
\rightarrow M$ centered at $C \subset {\rm Sing} (\mathcal{H})$, where ${\rm Sing} (\mathcal{H})$ stands for the singular set of
$\mathcal{H}$. The foliation $\mathcal{H}$ is said to be dicritical with respect to $\pi$ if its corresponding blow-up
$\widetilde{\mathcal{H}}$ does not leave the exceptional divisor $\pi^{-1} (C)$ invariant.
\end{defi}

Whenever no misunderstanding is possible, we will simply say that a given foliation is, or is not, dicritical without specifically mentioning
to the blow-up map. Also, for most of our discussion, it will suffice to consider blow-ups centered at single points (sometimes called
one-point blow-ups). For this type of blow-up, the characterization of $1$-dimensional dicritical foliations is very simple. More precisely,
let $\fol$ be a $1$-dimensional foliation on $(\C^n ,0)$ and fix a representative vector field $Y$ of $\fol$. Denote by $Y_k$ the
non-zero homogeneous component of least degree in the Taylor series of $Y$ based at $0 \in \C^n$. Then, we have:

\begin{lemma}
	\label{lemma1_dicritical}
The foliation $\fol$ is dicritical with respect to the blow-up centered at $0 \in \C^n$ if and only if $Y_k$ is a multiple
of the radial vector field $R = x_1 \partial /\partial x_1 + \cdots + 	x_n \partial /\partial x_n$.
\end{lemma}

\begin{proof}
	It suffices to compute the pull-back of $Y$ in the coordinates $(x_1, u_2, \ldots ,u_n)$ for the blow-up of $\C^n$ where
	the blow-up map $\pi$ is given by $\pi (x_1, u_2, \ldots ,u_n) = (x_1, x_1 u_2, \ldots , x_1u_n)$, cf. for example \cite{IY}
	or \cite{lectureNotes_RR}.
\end{proof}

In more general terms, a foliation $\fol$ is said to be {\it dicritical}\, at a center $C$ if there exists a sequence
of blow-ups beginning at $C$ and leading to a foliation which does not leave all the irreducible components of the global
exceptional divisor invariant (for details see Section~\ref{resolution_blowups}). Since a blow-up map is proper, and therefore
so is a composition of blow-up maps, there follows from Remmert's theorem that a leaf of the foliation $\widetilde{\fol}$ transverse
to (a component of) the exceptional divisor must project to a separatrix for the initial foliation $\fol$. Thus we obtain the following
simple characteristic of $1$-dimensional dicritical foliations:

\begin{lemma}
	\label{lemma2_dicritical}
	If the foliation $\fol$ at the center $C$ is dicritical, then the union of separatrices of $\fol$ through points of $C$
	yields a set with non-empty interior.\qed
\end{lemma}

The converse to Lemma~\ref{lemma2_dicritical} is known to hold for ambient spaces of dimension up to~$3$, and it is a simple
consequence of ``resolution theorems'', cf. Section~\ref{resolution_blowups}. Whereas the result is likely to hold in general,
a proof of this statement dispensing with ``resolution'' results seems to be still lacking in the literature.

The above lemmas show that $1$-dimensional dicritical foliations are, somehow, {\it very special}. In particular,
a ``generic foliation'' is not dicritical at their singular points. Also, owing to Lemma~\ref{lemma2_dicritical},
for most of the problems discussed here involving $1$-dimensional foliations, we can assume without loss of generality
that the foliation in question is not dicritical.

\begin{remark}
{\rm Examples of dicritical foliations are far more abundant when we consider {\it codimension~$1$} foliation in ambient spaces of
{\it dimension~$3$},. In particular,
it is unclear if there is any reasonable sense in claiming that a ``generic foliation'' is not dicritical. In fact, a good source
of examples of dicritical foliations consists of exploiting the affine Lie algebra generated by a homogeneous polynomial
vector field and by the radial vector field, see Section~\ref{core_dynamics}.}
\end{remark}

\subsection{Singularities of foliations on $(\C^2,0)$}\label{discussingdimension2}
As mentioned, singularities of foliations on $(\C^2, 0)$ are the object of a highly
developed theory, at least in the very general setting of {\it non-dicritical foliations}. In this paragraph, we shall collect some
reasons that allowed so much progress in this topic and compare them with the situation of foliations on $(\C^3, 0)$.

\bigbreak

\noindent {\bf (A) Seidenberg theorem}. It is commonly accepted that no general theorem in singularity theory can be proved
without relying on a suitable ``desingularization theorem''. In the theory of foliations, however, it is not possible in general
to actually {\it desingularize a foliation}, i.e., to obtain a non-singular model of the foliation up to
birational transformations. In fact,
whereas the phrase {\it desingularization theorem} is sometimes used as an abuse of language, a more accurate terminology would
be {\it reduction of singularities theorem}. In other words, rather than looking for a non-singular foliation, we look for a foliation
whose singular points are as ``well behaved as possible''. Typically, we will look for a foliation all of whose singular points are
{\it elementary}, i.e., all of them have {\it at least one eigenvalue different from zero}.

Seidenberg theorem \cite{seidenberg} provides a suitable procedure to reduce the singularities of holomorphic foliations
on $(\C^2, 0)$. Let $\fol$ denote a singular holomorphic foliation defined
on a neighborhood of $(0,0) \in \C^2$.
Seidenberg theorem asserts the existence of a finite sequence of
blow-up maps, along with transformed foliations $\fol_i$ ($i=1, \ldots, n$)
$$
\fol=\fol_0 \stackrel{\pi_1}\longleftarrow \fol_1 \stackrel{\pi_2}\longleftarrow \cdots
\stackrel{\pi_n}\longleftarrow \fol_n
$$
such that the following holds:
\begin{itemize}
	\item Each blow-up map $\pi_i$ ($i=1, \ldots, n$) is centered at a singular point of $\fol_{i-1}$.
	
	\item All singular points of $\fol_n$ are elementary, i.e., the foliation $\fol_n$ possesses at least one eigenvalue
	different from zero at each of them.
\end{itemize}

Denote by $D_1, \ldots , D_n$ the irreducible components of the total exceptional divisor associated with $\fol_n$.
Each $D_i$ is therefore a rational curve with strictly negative self-intersection and the corresponding Dynkin diagram is
a tree.

\bigbreak

\noindent {\bf (B) A global pseudogroup - Mattei-Moussu technique}.
Assume next that the foliation $\fol$ is not dicritical. Then, for each $i=1, \ldots, n$, $D_i \setminus {\rm Sing}\, (\fol_n)$
is a regular leaf of $\fol_n$, where ${\rm Sing}\, (\fol_n)$ stands for the singular set of $\fol_n$. In particular, all non-trivial
dynamics associated with the foliation $\fol_n$ is of transverse nature. Moreover this transverse dynamics naturally arises from the
holonomy representations of each of the leaves $D_i \setminus {\rm Sing}\, (\fol_n)$, $i=1, \ldots, n$. In turn, at least to a
considerable extent, the dynamics of these representations can be merged together through the argument of ``passage of corners''
(a.k.a. ``Dulac transform''), whenever $D_i \cap D_j \neq \emptyset$.

The preceding can be summarized by saying that all singular points of $\fol_n$ are ``dynamically connected''
in the sense that their local dynamics blend together in a nice pseudogroup of maps of $(\C,0)$. Furthermore,
the dynamics of this pseudogroup encodes virtually
all the information on the local structure of the initial foliation $\fol$.

The method described above to investigate the singularities of foliations on $(\C^2, 0)$ was very much set up in the seminal
paper by Mattei and Moussu \cite{mattei-moussu}. This technique has proven time and again to be extremely effective in a
variety of situations in dimension~$2$, see \cite{camachosad} and \cite{Guillot-Rebelo} for two examples of
problems whose solutions have involved this type of setup. In the next section we will discuss how far this approach can be
generalized to higher dimensions.

\bigbreak

\noindent {\bf (C) Dynamics of pseudogroups acting on $(\C,0)$}. Although for many problems this issue plays a relatively minor role,
let us still point out that the dynamics of pseudogroups acting on $(\C,0)$ is itself a highly developed topic. This type of dynamics was
first investigate by Huddai-Verenov \cite{Hudaiverenenov} and then by Il'yashenko in \cite{classicalilyashenko} where a ``generic situation''
of groups generated by hyperbolic diffeomorphisms was considered. In contrast, in \cite{mattei-moussu}, the authors have dealt with subgroups
all of whose orbits are finite. An absolute breakthrough then came with the works of Shcherbakov and of Nakai about general non-solvable
subgroups, see \cite{Sh1}, \cite{Sh2}, \cite{Na}. The reader may consult
\cite{Rebelo_ETDS} and references therein for a more complete
account of these dynamics in the non-solvable case whereas solvable pseudogroups are discussed in detail in \cite{rus}.

\begin{remark}
It should be pointed out that much progress in terms of construction of moduli spaces
for foliations on $(\C^2,0)$ and in describing the topology of leaves has
been made in recent years, chiefly by Mar\'{\i}n, Mattei, and Salem. While these aspects will not be discussed in this survey
which is mostly devoted to higher dimensional situations. Yet, the reader interested in the topology of leaves
will find more up-to-date information in \cite{mattei-1}, \cite{moremattei}, and \cite{Loic}. As to the construction
of moduli spaces, we refer to \cite{mattei-2} and to the preprints \cite{matteipreprints} and \cite{matteipreprintsagain}.
\end{remark}

\section{Splitting the problem: core dynamics and resolution}\label{core_dynamics}

The main object of this section are $1$-dimensional foliations defined around the origin of $\C^n$,
for $n \geq 3$. Yet, most of the discussion can be conducted without loss of generality in the case $n=3$.

It is useful to begin by recalling some well known facts about foliations on complex projective spaces. Let $\CP^n$ be viewed
as the space of radial lines through the origin of $\C^{n+1}$ and denote by $\Pi : \C^{n+1} \setminus \{ 0\} \rightarrow \CP^n$
the canonical projection. Also, for $\lambda \in \C^{\ast}$, denote by $h_{\lambda} : \C^{n+1}
\rightarrow \C^{n+1}$ the homothety defined by $h_{\lambda} (x_1, \ldots ,x_{n+1}) = (\lambda x_1, \ldots , \lambda x_{n+1})$.
Finally let $R$ denote the radial vector field $R = x_1 \partial /\partial x_1 + \cdots + x_{n+1} \partial /\partial x_{n+1}$ and
consider a homogeneous polynomial vector field
$$
X = P_1 \frac{\partial}{\partial x_1} + \cdots + P_{n+1} \frac{\partial}{\partial x_{n+1}}
$$
of degree~$d$ on $\C^{n+1}$. In other words, each $P_i$ is a degree~$d$ homogeneous polynomial, for every $i=1, \ldots ,n+1$. In
what follows $X$ is always assumed to satisfy the following conditions:
\begin{itemize}
	\item[(1)] The singular set of $X$ on $\C^{n+1}$ has codimension at least~$2$.
	
	\item[(2)] The vector fields $X$ and $R$ are linearly independent at generic points.
\end{itemize}
Next note that we have
$$
h_{\lambda}^{\ast} X = \lambda^{d-1} X
$$
so that the vector fields $h_{\lambda}^{\ast} X$ and $X$ are everywhere parallel for any fixed value of $\lambda \in \C^{\ast}$.
In particular, if $p \in \C^{n+1}$ is a point at which $X(p)$ and $R(p)$ are linearly independent, then $X(p)$ induces a direction
in $T_{q=\Pi (p)} \CP^n$ which is well defined in the sense that it does not depend on $p \in \Pi^{-1} (q)$. From this, it easily
follows that $X$ induces a singular holomorphic foliation $\fol$ on $\CP^n$ in the sense of Definition~\ref{1-dimensional_foliation_definition}.
A standard application of Serre's GAGA principle yields a type of converse for the above construction, namely
the following proposition holds, cf. for example \cite{IY}, \cite{lectureNotes_RR}.

\begin{prop}
\label{Naturefoliations_CPn}
Let $\fol$ denote a singular holomorphic foliation on $\CP^n$. Then, there exists a homogeneous polynomial vector field $X$
on $\C^{n+1}$ having singular set of codimension at least~$2$ and inducing the foliation $\fol$ on $\CP^n$ by means of the
above described construction.\qed
\end{prop}

Whereas, given $\fol$, the mentioned homogeneous vector field $X$ of Proposition~\ref{Naturefoliations_CPn} is not uniquely
defined, two homogeneous polynomial vector fields having singular set of codimension at least~$2$ and inducing the same
foliation on $\CP^n$ must have the same degree. Thus we can talk about the {\it degree of a foliation on $\CP^n$}\, as follows:

\begin{defi}
The degree of a foliation $\fol$ on $\CP^n$ is the degree of a homogeneous polynomial vector field on $\C^{n+1}$ having singular
set of codimension at least~$2$ and inducing $\fol$ in $\CP^n$ viewed as the space of radial lines of $\C^{n+1}$.
\end{defi}

Naturally blow-ups provide an alternative
way to realize the foliation induced on $\CP^n$ by a homogeneous polynomial vector field $X$ on $\C^{n+1}$.
Let $\widetilde{\C}^{n+1}$ stand for the blow-up of $\C^{n+1}$ at the origin and consider a homogeneous vector field
$X$ as above on $\C^{n+1}$. The blow up $\widetilde{X}$ of $X$ induces on $\widetilde{\C}^{n+1}$ the blow up $\tilf$ of the
the foliation $\fol$ induced by $X$ on $\C^n$.
Since, by assumption, $X$ is not everywhere parallel to the radial vector field $R$, there follows that the foliation
$\tilf$ leaves invariant the exceptional divisor $\pi^{-1} (0) \simeq \CP^n$. The restriction of $\tilf$ to the exceptional
divisor $\pi^{-1} (0) \simeq \CP^n$ can then naturally be identified with the foliation induced by $X$ on $\CP^n$ - viewed as
the space of radial lines of $\C^{n+1}$ - by means of the preceding construction.

Note that the blow up construction does not really requires the vector field to be homogeneous. In fact, as in Lemma~\ref{lemma1_dicritical},
consider a holomorphic vector field $Y$ defined around the origin of $\C^{n+1}$ whose Taylor series takes on the form
$Y = \sum_{i=k}^{\infty} Y_i$, where $Y_i$ stands for the homogeneous component of degree~$i$ of this Taylor series and $Y_k$
is not identically zero. As in Lemma~\ref{lemma1_dicritical}, we assume that $Y_k$ is not everywhere parallel to the radial vector field $R$.
The blow up of $Y$ induces a holomorphic foliation $\tilf$ on a neighborhood of the exceptional divisor $\pi^{-1} (0) \subset \widetilde{\C}^{n+1}$.
Moreover, since $Y_k$ is not a multiple of $R$, this foliation leaves $\pi^{-1} (0) \simeq \CP^n$ invariant and, in addition,
it is immediate to check that the restriction of $\tilf$ to $\pi^{-1} (0)$ coincides with the restriction to $\pi^{-1} (0)$ of
the foliation induced on $\widetilde{\C}^{n+1}$ by the blow up of $Y_k$ (alone). In particular,
the restriction of $\tilf$ to $\pi^{-1} (0)$
is identified with the foliation induced by the homogeneous vector field $Y_k$ on $\CP^n$ viewed
as the space of lines of $\C^{n+1}$.

The preceding motivates the following definition.

\begin{defi}\label{core_foliation}
	Let $\fol$ be a $1$-dimensional holomorphic foliation defined around the origin of $\C^n$ and assume that the blow up
	$\tilf$ of $\fol$ at the origin leaves the exceptional divisor $\pi^{-1} (0)$ invariant.
	Then the foliation induced on $\CP^{n-1} \simeq \pi^{-1} (0)$
	by the restriction of $\tilf$ is called the {\it core foliation}\, of $\fol$ and its global dynamics is referred to as the
	{\it core dynamics}\, of $\fol$.
\end{defi}

Again, if $Y$ is a representative of $\fol$ and $Y_k$ is as above ($Y = \sum_{i=k}^{\infty} Y_i$), the preceding then shows that the
core foliation of $\fol$ is nothing but the foliation induced on $\CP^{n-1}$ by the homogeneous vector field $Y_k$.

\subsection{$1$-dimensional foliations and dicritical codimension~$1$ foliations on $\C^3$}
The preceding discussion about foliations on projective spaces also applies to the case of codimension~$1$ {\it dicritical}\,
foliations on $(\C^n,0)$. To be more precise, codimension~$1$ dicritical foliation $\cald$ on $(\C^n,0)$ also induces through
the one-point blow-up centered at the origin a foliation on $\CP^{n-1}$, that will also be called the core foliation of $\cald$.
It is fair to say that this phenomenon and the corresponding dynamics were first exploited by Jouanolou \cite{jouanolou}
in his famous counterexample to a question posed by R. Thom. We shall review this issue below and go somewhat further by
exploiting the results in \cite{lorayandjulio} to see how difficult the situation may become.

In the sequel, we set $n=3$ to abridge notation.
First, let us characterize codimension~$1$ foliations that are {\it dicritical} for the blow-up of $\C^3$ centered
at the origin. Since the lemma below does not seem to be accurately stated in the literature, a detailed - albeit
straightforward - proof is included below.

\begin{lemma}
\label{invpunctual}
Assume that $\cald$ is a singular codimension~$1$ foliation defined on $(\C^3,0)$ and denote by $\tf$ its blow-up
centered at the origin. Then the exceptional divisor $\pi^{-1} (0) \simeq \CP^2$ is invariant under $\tf$ if and only
if no holomorphic vector field $Z$ tangent to $\cald$ admits a first non-zero homogeneous component (at the origin) that is
a multiple of the radial vector field $R$.
\end{lemma}

\begin{proof}
Let $\cald$ be given by a holomorphic $1$-form $\Omega = F \, dx + G \, dy + H \, dz$ whose singular set has codimension at least~$2$.
Denote by $\Omega_k$ the first non-zero homogeneous component of $\Omega$ at the origin, where $k$ stands for the degree
of $\Omega_k$. Next, let $\Omega_k = F^k  dx + G^k dy + H^k dz$. A direct inspection shows that $\pi^{-1} (0)$ {\it is not invariant}\,
by $\tf$ if and only if
\begin{equation}
xF^k + y G^k + z H^k =0 \, . \label{tang1revision}
\end{equation}
Now, note that if $Z$ is any vector field tangent to $\cald$, and whose first non-zero homogeneous component is denoted by $Z^{l}$, then
$Z^l$ naturally provides a solution for $\{ \Omega_k =0\}$, i.e., we have $\Omega_k . Z^l = 0$. However, if $Z^l$ happens to be a multiple
of the radial vector field, then $\Omega_k . Z^l = 0$ is tantamount to Equation~(\ref{tang1revision}) which is thus satisfied.
Hence, the exceptional divisor is not invariant by $\tf$.

To show that the existence of a vector field $Z$ satisfying the above mentioned conditions is also necessary, we proceed as
follows. Assume that $\cald$ is dicritical, i.e., that Equation~(\ref{tang1revision}) holds and denote
by $\wedge$ the standard exterior power of two vectors on $\C^3$. Next define
a vector field $v$ by letting $v(p) = R(p) \wedge (F(p), G(p), H(p))$ and then set $Z(p) = v(p) \wedge (F(p), G(p), H(p))$.
Clearly $Z$ is tangent to the foliation $\cald$. To complete the proof of the lemma,
it suffices to check that the first non-zero homogeneous component of
$Z$ at the origin is a multiple of the radial vector field. For this, note that we have
\begin{equation}
Z = ( zFH + yFG -x(H^2 + G^2) \, ,  xFG +zHG -y(F^2 + H^2) \, ,  yHG +xFH -z(G^2 + F^2)) \label{tang2revision}
\end{equation}
In particular the order of
$Z$ at the origin is at least $2k+1$. The homogeneous component of degree $2k+1$ is, in turn, given in vector notation by
$$
(xF^k + yG^k + zH^k) (F^k  , G^k  ,  H^k ) - ((F^k)^2 + (G^k)^2 + (H^k)^2)
(x ,   y  , z ) \; .
$$
In view of Equation~(\ref{tang1revision}), we conclude that the component of degree $2d+1$ of $Z$ at the origin
is given by $-((F^k)^2 + (G^k)^2 + (H^k)^2) R$.

To finish the proof of the lemma it suffices to show that the polynomial $(F^k)^2 + (G^k)^2 + (H^k)^2$ cannot vanish
identically. This, however, can easily be done by using the variables $(x,t,u)$ where the blow-up map becomes
$\Pi (x,t,u) = (x,xt,xu)$. In these variables, the dicritical condition (i.e. Equation~(\ref{tang1revision})) means that
$F^k (1,t,u) + tG^k(1,t,u) + uH^k(1,t,u)$ must vanish identically. Now, suppose for a contradiction that
$(F^k)^2 + (G^k)^2 + (H^k)^2$ is also identically zero. Then the two equations taken together imply that
$(t^2+1) (G^k)^2 + 2tu (G^k) (H^k) + (u^2+1) (H^k)^2$ vanishes identically as well. By solving the corresponding last equation for $G^k$,
we derive a contradiction with the fact that $G^k$ is itself a polynomial in the variables $t,u$. The lemma is proved.
\end{proof}

Next, let us consider again a homogeneous polynomial vector field $X$ on $\C^3$ satisfying conditions (1) and (2) in the previous
subsection, i.e. the singular set of $X$ has codimension at least~$2$ and the vector fields $X$ and $R$ are linearly independent
at generic points (note that in the case of homogeneous vector fields of degree at least $2$, conditions (1) and (2) are equivalent).
Since $X$ is homogeneous, we have
$$
[R,X] = (d-1) X
$$
where $d$ stands for the degree of~$X$. Thus the pair $X$ and $R$ generates the Lie algebra of the affine group. In particular
the distribution of planes (of dimension~$2$) spanned by $X$ and $R$ is involutive and hence integrable. Let us then denote by
$\cald$ the codimension~$1$ foliation spanned by $X$ and $R$.

Let $\tf$ stands for the blow-up of $\cald$ centered at the origin so that $\tf$ is defined on $\widetilde{\C}^3$. Owing to
Lemma~\ref{invpunctual}, the exceptional divisor $\pi^{-1} (0) \simeq \CP^2$ {\it is not invariant}\, under $\tf$. Furthermore,
the structure of the foliation $\tf$ (and hence that of $\cald$) is essentially as complicated as the structure of the blow-up
$\tilf$ of $\fol$, where $\fol$ denotes the foliation induced by $X$. This observation deserves further comments.

To begin with, recall that $\widetilde{\C}^3$ can also be seen as the tautological line bundle over $\pi^{-1} (0) \simeq \CP^2$.
The bundle projection will be denoted by $\widetilde{\Pi} : \widetilde{\C}^3 \rightarrow \pi^{-1} (0)$ since it can naturally
be identified with the canonical projection $\Pi : \C^{3} \setminus \{ (0,0,0)\} \rightarrow \CP^2$. Next, recall that, unlike
$\tf$, the foliation $\tilf$ leaves the exceptional divisor invariant. In particular, at a point $p$ of $\pi^{-1} (0) \simeq \CP^2$
that is regular for $\tilf$, this foliation defines a direction $u_p \in T_p \pi^{-1} (0)$. Next, for $p$ ``sufficiently generic'',
the leaf of $\tf$ intersects transversely $\pi^{-1} (0)$. This transverse intersection naturally defines a direction
$v_p \in T_p \pi^{-1} (0)$. It is immediate to check that the directions of $v_p$ coincides with the one defined by
$\tilf$. Denoting by $\tilf_{\vert \pi^{-1} (0)}$ the foliation on $\pi^{-1} (0)$ obtained by restriction of $\tilf$,
we have the following:

\begin{lemma}\label{example_homogeneousvectorfield+radial}
The leaves of the foliation $\tf$ are of the form $\widetilde{\Pi}^{-1} (L)$ where $L$ is a leaf of $\tilf_{\vert \Pi^{-1} (0)}$.
Similarly, every leaf of $\tf$ is invariant by the foliation $\tilf$.\qed
\end{lemma}

Recalling that every foliation on a projective space is induced by a homogeneous polynomial vector field, the
interest of Lemma~\ref{example_homogeneousvectorfield+radial} is actually captured by the following slightly loose statement:
{\it every foliation on $\CP^2$ is naturally the core foliation for singularities of both dimension~$1$ and codimension~$1$
foliations on $(\C^3,0)$}.

Before considering some concrete applications of the previous remark, let us close this section by
pointing out that the above construction allows us to define the core of a dicritical codimension~$1$ foliation on $(\C^3,0)$
as follows.

\begin{defi}\label{core_foliation-codimension1_C3}
Let $\cald$ be a codimension~$1$ holomorphic foliation defined around the origin of $\C^3$ and assume that the blow up
$\tf$ of $\cald$ at the origin does not leave the exceptional divisor $\pi^{-1} (0)$ invariant.
Then the foliation induced on $\CP^{n-1} \simeq \pi^{-1} (0)$
by the restriction of $\tilf$ is called the {\it core foliation}\, of $\fol$ and its global dynamics is referred to as the
{\it core dynamics}\, of $\fol$.
\end{defi}


\subsection{Jouanolou's example, chaotic dynamics, and their meaning for singularity theory}\label{subsec:Jouanolou}
Let us go back to R. Thom's question
on the existence of separatrices for codimension~$1$ foliations on $(\C^3,0)$, cf. Definition~\ref{definition_separatrix}. As pointed
out in Section~\ref{basics}, it is not always easy to construct codimension~$1$ foliations due to Frobenius integrability condition
that has to be satisfied by the distribution of planes in question. Yet, the discussion revolving around
Lemma~\ref{example_homogeneousvectorfield+radial} also indicates a simple way to construct lots of {\it dicritical}\, codimension~$1$
foliations on $\C^3$. More precisely, every foliation on $\CP^2$ yields one such dicritical codimension~$1$ foliation.

Let then $\cald$ be a dicritical codimension~$1$ foliation as above and assume that $\cald$ admits separatrices. Let then $S$ denote a
germ of an irreducible separatrix for $\cald$. Since $S$ has codimension~$1$, there follows the existence of a germ of an irreducible
holomorphic function $f : (\C^3,0) \rightarrow (\C,0)$ such that $S$ coincides with the set $\{ f=0\}$. In terms of Taylor series, we
set $f= \sum_{i \geq l}^{\infty} f_i$ where $l$ is the degree of the first non-zero homogeneous component of $f$. Let $\mathcal{C}
\subset \CP^2$ be the curve defined on the projective plane by the homogeneous equation $\{ f_l = 0\}$ (the tangent cone to $S$).
If we denote by $\fol$ the {\it core}\, of $\cald$ (recall than that $\fol$ is a $1$-dimensional foliation on $\CP^2$), then the
following can be said:

\begin{lemma}\label{Invariantcone_remarkjouanolou}
	With the preceding notation, the curve $\mathcal{C} \subset \CP^2$ is invariant by $\fol$.
\end{lemma}

\begin{proof}
The foliation $\cald$ is defined by an integrable $1$-form $\Omega$ whose Taylor series takes on the form
$\Omega = \sum_{i=k}^{\infty} \Omega_i$ where $k$ stands again for the first non-zero homogeneous component of $\Omega$.
A simple argument based on degrees shows that the $1$-form $\Omega_k$ is integrable as well, i.e., it satisfies Frobenius equation
$\Omega_k \wedge d\Omega_k =0$. Similarly, one checks that the (homogeneous) surface defined by $\{ f_l =0\}$ yields a separatrix
for the codimension-$1$ foliation $\cald_k$ induced by $\Omega_k$. Set $\Omega_k = F^k  dx + G^k dy + H^k dz$ and, as usual, let $R$ denote
the radial vector field on $\C^3$.

Next recall that a homogeneous vector field of $\C^3$ representing $\fol$ is well defined only up to a multiplicative constant
and addition of a multiple of the radial vector field. Now, since $\Omega_k$ is homogeneous, the vector $R (p)$ is contained in the
plane defined by the kernel of $\Omega_k (p)$ at the point $p$. Hence, up to eliminating multiplicative factors, a representative
vector field $X$ for $\fol$ can be obtained by letting $X (p) = R(p) \wedge (F^k(p), G^k(p), H^k(p))$. In particular $X(p)$ lies
in the kernel of $\Omega_k (p)$, i.e., $X$ is tangent to the foliation $\cald_k$. Finally, since at regular points $p \in \{ f_l =0\}$
the tangent space at $\{ f_l =0\}$ and the kernel of $\Omega_k (p)$ coincide, we conclude that $X$ is tangent to the surface $\{ f_l =0\}$.
The lemma then follows immediately.
\end{proof}

In view of Lemma~\ref{Invariantcone_remarkjouanolou}, the basic remark of Jouanolou concerning Thom's conjecture was the following one:
if we can find a foliation $\fol$ on $\CP^2$ leaving invariant no algebraic curve, then the (dicritical) codimension~$1$
foliation $\cald$ arising from combining the radial vector field of $\C^3$ and a representative of homogeneous vector field for $\fol$ will
admit {\it no separatrix}.

Jouanolou's remark is possibly the first instance where the existence of the {\it core dynamics}\, actually impacts
the study of singularities of foliations. From this point of view, the main result of Jouanolou in \cite{jouanolou-2} can be stated
as follows:

\begin{theorem}\cite{jouanolou-2}
For every $d \geq 2$, the foliation induced on $\CP^2$ by the vector field	
$$
X_d = y^d \partx + z^d \party + x^d \partz \,
$$
leaves no algebraic curve invariant.
\end{theorem}

Jouanoulou theorem implies, in particular, that for every fixed $d \geq 2$, there exist foliations of degree~$d$ that are not
tangent to any algebraic curve of $\CP^2$.

Armed with the above theorem, there follows from what precedes that the codimension~$1$ {\it Jouanolou foliation $J_d$}, $d \geq 2$,
of $\C^3$ which is defined as the singular foliation spanned by $X_d$ and the radial vector field is a {\it counterexample
to Thom's question}. The well-known explicit $1$-form $\Omega$,
$$
\Omega = (yx^d -z^{d+1}) \, dx \; + \; (zy^d - x^{d+1}) \, dy \; + \;
(xz^d - y^{d+1}) \, dz \; ,
$$
defining the foliation $J_d$ can promptly be obtained by taking the vector product of $X_d$ and $R$.

\bigbreak

The next question is to wonder how far the {\it core dynamics} can influence the study of singularities of foliations,
say of dimension~$1$ on $\C^n$, $n\geq 3$. In other words, owing to the discussion in this section, the detailed understanding
of the local structure of one such foliation arguably passes through the global description of its {\it core foliation}.
This understanding would require, in particular, a (global) control of the dynamics of the core foliation. At this point, we might
wonder whether it is possible to obtain such an accurate local description of all $1$-dimensional foliations on, say, $(\C^3,0)$.
From the standpoint emphasized above, an easier question would be to provide a reasonable global description of all or nearly all foliations
on $\CP^2$. Unfortunately, the latter question does not seem to admit an affirmative answer as it follows
from Loray-Rebelo theorem \cite{lorayandjulio} as stated below.

Fix positive integers $n$ and $d$, with $\min \{ n,d \} \geq 2$. A straightforward counting of parameters shows that the space ${\rm Fol}^{(d)(n)}_{\CP}$
of degree~$d$ foliations on $\CP^n$ can be identified with a Zariski-open set of the complex projective space of dimension
$$
(d+n+1){(d+n-1)!\over d!(n-1)!}-1.
$$
This space of foliation can then be furthered moduled out by the action of the automorphism group
${\rm PSL}\, (n+1, \C)$ of $\CP^n$ but this will not be needed in the sequel.
The main upshot here is that ${\rm Fol}^{(d)(n)}_{\CP}$ can be parameterized by a finite dimensional complex manifold and, in particular,
inherits of a natural topology. With this notation, the main result of \cite{lorayandjulio} reads as follows:

\begin{theorem}\cite{lorayandjulio}\label{FrankandJulio}
Fixed $n,d \geq 2$, there exists a non-empty open subset $\mathcal{U} \subset {\rm Fol}^{(d)(n)}_{\CP}$ such that every foliation
$\fol$ lying in $\mathcal{U}$ satisfies all the conditions below:
\begin{enumerate}
	\item All singular points of $\fol$ are hyperbolic. In particular, they form a finite set.
	
	\item Minimality: Every leaf of $\fol$ is dense in $\CP^n$.
	
	\item Ergodicity: Every measurable set of leaves has either zero or total
	Lebesgue measure.
	
	\item Rigidity: If $\fol' \in {\rm Fol}^{(d)(n)}_{\CP}$ is conjugate to $\fol$ by a homeomorphism
	$h : \CP^n \rightarrow \CP^n$ that is close to the identity, then $\fol$ and $\fol'$ are also conjugate by an element of
	${\rm PSL}\, (n+1, \C)$.
\end{enumerate}
\end{theorem}

The level of dynamical complication exhibited by the foliations indicated above puts any accurate description of them basically
out of reach. Moreover, even up to topological conjugation, it is not possible to achieve a reasonable list of ``models'' or
``normal forms'' owing to the above indicated rigidity phenomenon.

It is convenient to expound a bit on the consequences of Theorem~\ref{FrankandJulio} from the point of view of singularity
theory for $1$-dimensional foliations on dimensions~$3$ and greater. Consider then a foliation lying in the set
$\mathcal{U} \subset {\rm Fol}^{(d)(n)}_{\CP}$ provided by Theorem~\ref{FrankandJulio}. As a foliation defined
on $\CP^n$, it can be represented by some homogeneous polynomial vector field $X$ on $\C^{n+1}$. In other words,
if $\fol$ is the foliation on $\C^{n+1}$ induced by the local orbits of $X$ then $\tilf$, its (one-point)
blow-up at the origin, leaves the exceptional divisor $\pi^{-1} (0) \simeq \CP^n$ invariant and is such that
the restriction of $\tilf$ to $\pi^{-1} (0) \simeq \CP^n$ is naturally identified with the initial foliation
in $\mathcal{U} \subset {\rm Fol}^{(d)(n)}_{\CP}$.

Now recall that the vector field $X$ is not uniquely defined: most notably, we can add to $X$ any multiple of
the radial vector field by a homogeneous polynomial of degree~$d-1$. Since, in addition,
the singularities of the initial foliation in
$\mathcal{U}$ are all hyperbolic, it is easy to conclude that the vector field $X$ can be chosen so as to fulfill the
following conditions:
\begin{enumerate}
	\item The foliation $\fol$ has an isolated singularity at the origin of $\C^{n+1}$.
	
	\item The foliation $\tilf$, viewed
	as foliation on a manifold of dimension~$n+1$, still have only hyperbolic singularities.
\end{enumerate}
Furthermore, a generic
choice of the initial foliation in $\mathcal{U}$ and of the vector field $X$ allows us to rule out the existence
of resonances at the singular points of $\tilf$ as well. Thus, all the singularities of $\tilf$ are, in fact, linearizable.
Also, all the above mentioned characteristic are stable under higher order perturbations of a representative vector field.
The situation can then be summarized as a statement in itself.

\begin{theorem}
	\label{summary_chaosinfoliations} {\rm (Corollary of~\cite{lorayandjulio})}
	For every degree $d\geq 2$, there exists a non-empty open set $V$ of homogeneous vector fields of degree~$d$ in $\mathfrak{X} (\C^{n+1},0)$
	such that every germ of foliation $\fol$ represented by a holomorphic vector field $X$ having the form $X = X^d + {\rm h.o.t.}$,
	with $X^d \in V$ and where ${\rm h.o.t.}$ stands for higher order terms, satisfy all of the following conditions:
	\begin{itemize}
		\item[(1)] The one-point blow up $\tilf$ of $\fol$ at the origin leaves the exceptional divisor
		$\pi^{-1} (0) \simeq \CP^n$ invariant.
		
		\item[(2)] All singular points of $\tilf$ are hyperbolic and linearizable. In particular, $\tilf$ has exactly
		$$
		\frac{d^{n+1} -1}{d-1}
		$$
		singular points and all of them lie in $\pi^{-1} (0) \simeq \CP^n$.
		
		\item[(3)] The restriction of $\tilf$ to $\pi^{-1} (0) \simeq \CP^n$ defines a degree~$d$ foliation
		of $\CP^n$ lying in the open set $\mathcal{U} \subset {\rm Fol}^{(d)(n)}_{\CP}$ given by
		Theorem~\ref{FrankandJulio}.
	\end{itemize}
\end{theorem}

Let us point out that the formula in item~(2) for the number of singular points of $\tilf$, i.e., for a degree~$d$
foliation on $\CP^n$ all of whose singular points are hyperbolic is well known and can be proved in a variety of
ways. For example, by choosing affine coordinates yielding a ``hyperplane at infinity'' on which the foliation
has no singular point and then applying B\'ezout theorem to the corresponding polynomial vector field representing
the foliation in the above indicated affine coordinates.

\begin{remark}
	Naturally the content of Theorem~\ref{summary_chaosinfoliations} can be adapted to germs of codimension~$1$
	dicritical foliations on $(\C^3,0)$.
\end{remark}

To close this section, it is convenient to make a parallel with the discussion in Section~\ref{discussingdimension2}
for singularities of foliations on $(\C^2,0)$ so as to better appreciate the difficulties arising from the existence
of wild core dynamics as stated in Proposition~\ref{summary_chaosinfoliations}.

\bigbreak

\noindent {\bf (A') Generalizations of Seidenberg theorem to $(\C^n,0)$}. The problem is wide open for $n\geq 4$
though sharp desingularization theorems are now established for $n=3$. The topic is of undisputed interest since
virtually all general statements about
singularities rely, directly or indirectly, on a suitable ``resolution theorem''. Yet, for $n \geq 3$, the ability
to obtain a model of the foliation where all singular points are ``simple enough'' might still be a long way
off of providing an accurate description of the singularity in question.

To substantiate the above claim, it suffices to consider the local foliations $\fol$ on $(\C^n,0)$ provided
by Theorem~\ref{summary_chaosinfoliations}. The blow-up $\tilf$ of $\fol$ at the origin provides a birational
model for $\fol$ possessing only ``simple singular points'': in fact, all singularities of $\tilf$ are hyperbolic
and linearizable. In other words, the local behavior of $\tilf$ around each of its singular points is essentially
trivial and promptly available. The very complicated dynamical behavior of $\fol$ around the origin is, however,
encoded in its core dynamics (cf. Lemma~\ref{example_homogeneousvectorfield+radial}) but the global nature
of the core dynamics prevents resolution theorems to yield any insight into this dynamical system.

\bigbreak

\noindent {\bf (B') Taming the core dynamics}. If one is to fully understand the structure of a foliation around a
singular point, then an accurate description of its core dynamics needs to be envisaged. If Proposition~\ref{summary_chaosinfoliations}
tells us this is a kind of unrealistic goal, it also raises the question of ``selecting'' those classes of singular
points allowing a more detailed description. This is a very interesting point as it hints at considering the connections between
singularity theory and the remainder of Mathematics or, even, Physics. Singularities playing a special role
in problems from Geometry, Complex Analysis and/or Integrable Systems are likely to be amenable to a more complete analysis.
Examples of these situations will be discussed in the forthcoming sections.

In terms of ``taming core dynamics'', of course the ideal situation would be to have a core foliation defining an ``integrable
system'' in some suitable sense. Alternatively, for a number of problems, it might be enough to ensure the existence of
(``sufficiently many'') algebraic invariant curves. An important issue involving invariant curves is that more often than not
the dynamics of the foliation in question can be investigated in more details on a neighborhood of them, especially when their
fundamental group contains more than a single generator. This study, whereas of more global nature, is somehow akin to
``Mattei-Moussu pseudogroup technique'' mentioned in Section~\ref{discussingdimension2}. Interesting examples where this point
of view have successfully been employed - even outside the scope of singularity theory - include \cite{classicalilyashenko},
\cite{lorayandjulio}, and \cite{Guillot-Rebelo}.

\bigbreak

\noindent {\bf (C') Dynamics of pseudogroups acting on $(\C^n,0)$}. The perspective of focusing in the local dynamics
arising from the holonomy of an invariant algebraic curve in higher dimensions naturally leads us towards considering the dynamics
of subgroups of ${\rm Diff}\, (\C^n,0)$, $n \geq 2$. As was to be expected, many new dynamical phenomena arise for $n\geq 2$ compared to the
situation $n=1$. As pointed out in Section~\ref{discussingdimension2} much is known about the dynamics of subgroups
of ${\rm Diff}\, (\C,0)$ whereas for subgroups of ${\rm Diff}\, (\C^n,0)$, the theory is still in its early stages.

Nonetheless, we mention that generalizations to higher dimensions of Mattei-Moussu's theorem on groups with finite orbits is by now
well understood, see \cite{bolSBM}, \cite{diffC2}, \cite{Ribon}. These results are likely to have impact in problems about
existence of first integrals but they might also provide insight in higher dimensional versions of the so-called ``analytic
limit set'', see \cite{camacho}.

Finally a major issue in the theory is to find sharp conditions to extend to higher dimensions the Shcherbakov-Nakai theory
of local vector fields in the ``closure of the group'' \cite{Sh1}, \cite{Sh2}, \cite{Na}. Very little is known about this question
aside from some results in \cite{lorayandjulio} which rely on the existence of a hyperbolic contraction for the group in question.
This assumption looking rather far from sharp, the topic appears to be ripe for significant progress.


\section{Resolution theorems in dimension~$3$}\label{resolution_blowups}

In the remainder of this survey we will discuss relatively recent progress in some of the several aspects of
singularity theory. This section is devoted to ``resolution theorems'' while the next two sections will basically review
the general problem of invariant varieties and the study of a particular and important class of singular points,
namely the semicomplete ones. In the course of these discussions
theorems providing - at various degrees - some control on the core dynamics in question will play a prominent role.

As previously indicated, theorems on reductions of singular points are always of paramount importance in the theory whether
or not there are major difficulties lying out of their reach (e.g. complicated core dynamics). For foliations defined
on complex $2$-dimensional manifolds (or varieties), Seidenberg's theorem provides a sharp {\it reduction of singularities
theorem}\, (a.k.a. ``resolution theorem'') that is particularly easy to manipulate. Beyond dimension~$2$, decisive results
exist only in dimension~$3$, where it is already necessary to distinguish between foliations of dimension~$1$ and foliations
of codimension~$1$. This section is devoted to reviewing and explaining the main ``resolution theorems'' for
$1$-dimensional foliations in dimension~$3$.

Owing to the classical Hironaka resolution theorem, we can assume that our singular foliations are always defined on manifolds
(i.e. smooth complex spaces). Furthermore, since the problems are local, we may assume them to be defined on a neighborhood
of the origin of $\C^n$. The case $n=2$ being settled by the above mentioned theorem of Seidenberg, we assume from now on that
$n=3$, i.e., our foliations are defined on a neighborhood of the origin of $\C^3$.

Working on $(\C^3,0)$, we need to distinguish between foliations of dimension~$1$ and foliations of codimension~$1$. The case
of codimension~$1$ foliations was settled earlier in \cite{cano}. However, the story involving
foliations of dimension $1$ - the main object of this survey - is longer and more elusive.

Resolution results for foliations of dimension 1 on $(\C^3,0)$ have first appeared in \cite{canoLNM}, where the author proves
the existence of a formal local uniformization theorem. In this work, the author also hints at the existence of a
new phenomenon involving
singularities possessing a certain formal separatrix (i.e. a formal curve invariant by the foliation) which posed serious difficulties
to resolve the singularity by means of standard blow ups. The issue was made clear by Sancho and Sanz who
provided explicit examples of
foliations in $(\C^3,0)$ that cannot be reduced by sequences of standard blow-ups centered at sets contained
in the singular loci of the initial foliations and its transforms.

After the examples found by Sancho and Sanz, the next truly major result in the area is due to D. Panazzolo \cite{daniel2}. In \cite{daniel2},
Panazzolo considers singularities of real foliations in (real) dimension~$3$. He works in the real setting,
rather than in the complex one, mostly due to the fact that his original motivation lied in Hilbert's problem
about the number of limit cycles of a polynomial vector field on
$\R^2$. In his work, Panazzolo shows that the corresponding germs of foliations can always be turned into a
foliation all of whose singular points are elementary {\it by means of a finite sequence of weighted blow ups
centered at singular sets}. The proof is constructive and actually provides
an algorithm to obtain the desired reduction of singularities. Later, relying on Panazzolo's algorithm
introduced in~\cite{daniel2},
McQuillan and Panazzolo were able to provide
a very satisfactory answer to the generalization of Seidenberg's theorem for foliations on $(\C^3,0)$ in \cite{danielMcquillanpreprint}, \cite{danielMcquillan}.

The preprint \cite{danielMcquillanpreprint} was made available in 2007 and a few years later,
Cano, Roche, and Spivakovsky revisited the topic
from the point of view of valuation theory, see \cite{C-R-S}.
Their strategy is in line with Zariski's general approach to desingularization problems and, hence, is
essentially divided in two parts. First, for a given foliation, we seek to ``simplify''
only the singularities lying in the center of a given valuation
(identified with its transforms, or extensions, through blow-ups). Resolution results for singularities lying in the center of
a valuation are often referred to as {\it local
uniformization theorems} and the first part of Zariski approach aims at obtaining this type of statement.
Once a convenient local uniformization result is obtained, the second part of Zariski approach deals
with its {\it globalization}. More precisely,
once it is proved that for every valuation $\nu$, the singularities lying in the center of $\nu$ can be
simplified (in some appropriate sense),
we try to conclude that, in fact, all singularities of the foliation can simultaneously be simplified in the same sense.
When it comes to applying this point of view to singularities of foliations most of the difficulties related to
the globalization procedure are handled pretty well a very general
(axiomatic) {\it gluing theorem} due to O. Piltant \cite{Piltant}. Owing to Piltant's theorem, it is fair to say
that the fundamental difficulty of resolution problems for foliations, in arbitrary dimensions, revolves
around obtaining suitable local uniformization theorems.

In view of what precedes, the content of \cite{C-R-S} can roughly be summarized by claiming
{\it the existence of a birational model for the initial foliation where the singular points
are log-elementary}. The reader is referred to \cite{C-R-S} for the definition of log-elementary singular points.
For our purposes, it suffices to know that such singularities are, at worst, {\it quadratic} in the sense that
they are locally given by a representative vector fields with non-zero second-jet at the singular point in question.
One of the goals of \cite{helenajulio_RMS} was to complete the work of Cano-Roche-Spivakovsky by deriving
``final models'' similar to those of \cite{danielMcquillan}, in order to obtain
a global resolution theorem comparable to \cite{danielMcquillanpreprint},
\cite{danielMcquillan} through Zariski classical approach.

We will compare the resolution theorems for foliations obtained by McQuillan-Panazzolo in~\cite{danielMcquillan} and by ourselves
in \cite{helenajulio_RMS}, they correspond to Theorem~$2$ and Theorem~A of the respective papers. In particular, it will be seen
that in the context of foliations the two results are pretty much
equivalent and can be summarized by the following assertion: {\it given a singular holomorphic $1$-dimensional foliation $\fol$ on
$(\C^3,0)$, there exists a birational model of $\fol$ where all singular points
are elementary}. In this sense, the only difference between the theorems in question will be down
to the way in which the desired birational model is constructed.


\subsection{Persistent nilpotent singularities}\label{subsec:pns}

As already mentioned, Sancho and Sanz have showed the existence of foliations in $(\C^3,0)$ that cannot be reduced by
sequences of standard blow-ups with centers contained in
the singular set of the initial foliation and its transforms. In fact, their result is slightly more general in the sense
that we may allow for blow-ups of invariant centers not necessarily contained in the singular locus. As a matter
of fact, they have provided a $3$-parameter family of foliations whose elements cannot be turned into a foliation
all of whose singularities are elementary by means of blow ups centered in the singular loci and whose generic element
cannot be turned into a foliation with elementary singular points even if invariant centers are allowed.
This family of foliations is represented by the family of vector fields $X_{\alpha,\beta,\lambda}$ taking on
the form
\begin{equation}\label{example_Sancho_Sanz}
X_{\alpha,\beta,\lambda} = x\left( x \frac{\partial}{\partial x} - \alpha y \frac{\partial}{\partial y}  - \beta z \frac{\partial}{\partial z}\right)
+ xz \frac{\partial}{\partial y} + (y-\lambda x) \frac{\partial}{\partial z} \, .
\end{equation}
Accordingly, foliations in this family will be denoted by $\fol_{\alpha,\beta,\lambda}$.
The foliations $\fol_{\alpha,\beta,\lambda}$ are {\it nilpotent} at the origin in the sense that so are the vector fields
$X_{\alpha,\beta,\lambda}$. For reference, it is convenient to make accurate the notion of {\it nilpotent foliation}.

\begin{defi}\label{nilpotentsingularpoint_foliation}
A ($1$-dimensional) holomorphic foliation is said to have a nilpotent singularity at a singular point $p$ if
its representative vector field around~$p$ has non-zero nilpotent linear part at~$p$.
\end{defi}

The above notion of nilpotent singularity is well defined since it does not depend on the choice of the representative
vector field. Also, whenever no misunderstanding about the singular point in question is possible, we will
abridge notation by simply saying that
{\it $\fol$ is a nilpotent foliation}. Going back to the nilpotent foliations $\fol_{\alpha,\beta,\lambda}$, we note that the
plane $\{x=0\}$ is invariant by them and that it contains
the singular set of $\fol_{\alpha,\beta,\lambda}$ which coincides with the axis $\{x = y = 0\}$. Recalling that
a singular point is said to be {\it elementary} if the representative vector field possesses at least one eigenvalue
different from zero, we now have the following:

\begin{prop}\label{SanchoSanz_examples}
The foliations in the family $\fol_{\alpha,\beta,\lambda}$ cannot be turned into a foliation all of whose singular points
are elementary by means of a sequence of standard blow-ups with centers contained in singular sets.
\end{prop}

\begin{proof}[Sketch of Proof]
Consider the one-point blow-up centered at the origin of $(\C^3,0)$ and let $\pi$ stands for the blow-up map. Let then $(x,u,v)$ be the
affine coordinates for the blown-up space where $y = ux$ and $z = vx$. The pull-back $\pi^{\ast} X_{\alpha,\beta,\lambda}$ of the
vector field $X_{\alpha,\beta,\lambda}$ is given by
\[
\pi^{\ast} X_{\alpha,\beta,\lambda} = x\left( x \frac{\partial}{\partial x} - (\alpha + 1) u \frac{\partial}{\partial u}  - (\beta + 1) v
\frac{\partial}{\partial v}\right) + xv \frac{\partial}{\partial u} + (u - \lambda) \frac{\partial}{\partial v} \, ,
\]
whose expression is similar to the expression of $X_{\alpha,\beta,\lambda}$. In fact, the main difference between the two expressions
concerns the last term. Note, however, that the origin of the present coordinates is not contained in the singular set of the induced
foliation, which is given by $\{x=0 , u=\lambda\}$. Thus, if we consider the translation $T(x,\overline{u}, \overline{v}) = (x,
\overline{u} + \lambda, \overline{v} + \mu)$, the pull-back of $\pi^{\ast} X_{\alpha,\beta,\lambda}$ through $T$ is given by
\[
x\left( x \frac{\partial}{\partial x} - (\alpha + 1) \overline{u} \frac{\partial}{\partial
\overline{u}}  - (\beta + 1) \overline{v} \frac{\partial}{\partial \overline{v}}\right) + x(\overline{v} + \mu - \lambda(\alpha + 1))
\frac{\partial}{\partial \overline{u}} + (\overline{u} - \mu(\beta + 1)x) \frac{\partial}{\partial \overline{v}} \, .
\]
In the particular, if we choose $\mu = \lambda(\alpha + 1)$, the vector field in question coincides with the vector field
$X_{\alpha+1,\beta+1,\lambda(\alpha + 1)(\beta+1)}$. In other words, the
transformed foliation of $\fol_{\alpha,\beta,\lambda}$ contains
a nilpotent singular point belonging to the (initial) Sancho-Sanz family. It can be checked that the
same issue occurs if the blow-up centered at the curve of singular points of $\fol_{\alpha,\beta,\lambda}$
is considered.
\end{proof}

Summarizing what precedes, every sequence of blow ups as above applied to a foliation in Sancho-Sanz family
lead to a foliation having a singular point where the foliation is locally conjugate to another foliation
in the initial family. In particular, all transformed foliations will exhibit a nilpotent singular point. This
nilpotent singular point has a geometric interpretation naturally related to the issues raised
by Cano in \cite{canoLNM} for a
resolution by standard blow-ups. In fact, by elaborating in the above indicated argument,
Sancho and Sanz have shown that the parameters $\alpha,\beta,\lambda$ can be chosen so that the
foliation associated with the vector field $X_{\alpha,
\beta,\lambda}$ possesses a {\it strictly formal separatrix $S = S_0$ through the origin}. Moreover, given a sequence of blow ups
as before, the sequence of points $\{ p_n\}$ in the exceptional divisors corresponding to the position of the
(persistent) nilpotent singularity is determined by the sequence of transforms $\{ S_n \}$ of the formal separatrix $S=S_n$.
We should still note that, the fact that
every separatrix $S_n$ is stricty formal says that even in the case we allow blow-ups to be centered at
analytic invariant curves that {\it are not}\, necessarily contained in singular set of the foliation, a
resolution procedure still does not exist.

In terms of the relation between foliations and - possibly formal - separatrices, a natural object that plays an important
role is the notion of {\it multiplicity of the foliation along the separatrix}. Let $X$ be a representative vector
field of $\fol$ and $\varphi$ the Puiseux parametrization of $S$. Let $\varphi^{\ast} X|_S$ stands for the pull-back of the restriction of
$X$ to $S$. If $\varphi^{\ast} X|_S = g(t) \partial /\partial t$, then the multiplicity of $\fol$ along $S$ is defined as the order of $g$
at $0 \in \C$ (assuming that $\varphi(0)$ coincides with the singular point).

In \cite{helenajulio_RMS}, we introduced the notion of {\it persistent nilpotent singular point} which is as follows.

\begin{defi}
A nilpotent singular point $p_0$ of a foliation $\fol_0$ is said to be persistent
if there exists a formal separatrix $S_0$ for $\fol_0$ through $p_0$ such that for every sequence of blowing-ups
\[
\fol_0 \stackrel{\pi_1} \longleftarrow \fol_1 \stackrel{\pi_2} \longleftarrow \cdots
\stackrel{\pi_l}\longleftarrow \fol_n \,
\]
where $\fol_i$ stands for the transformed of $\fol_{i-1}$ through the (standard) blow-up centered at some $C_{i-i} \subseteq {\rm Sing} \,
(\fol_{i-1})$ containing the point $p_{i-1}$ (selected by the transformed separatrix $S_{i-1}$, in the sense that it corresponds to the intersection
of $S_{i-1}$ with the excetional divisor), the following conditions are satisfied
\begin{itemize}
\item[(a)] the singular points $p_i$ are all nilpotent singular points for the corresponding foliations;
\item[(b)] the multiplicity of $\fol_i$ along $S_i$ does not depend on $i$.
\end{itemize}
\end{defi}

The multiplicity of a foliation $\fol$ along a separatrix (possibly a formal one) $S$ is defined as follows. Let $X$ be a representative vector
field of $\fol$ and $\varphi$ the Puiseux parametrization of $S$. Let $\varphi^{\ast} X|_S$ stands for the pull-back of the restriction of
$X$ to $S$. If $\varphi^{\ast} X|_S = g(t) \partial /\partial t$, then the
{\it multiplicity of $\fol$ along $S$}, ${\rm mult} (\fol, S)$, is defined as the order of $g$
at $0 \in \C$ (assuming that $\varphi(0)$ coincides with the singular point).

The role played by the notion of multiplicity of a foliation along a separatrix is closely related to its natural behavior
under blow ups. Recall that the {\it order of a foliation at a singular point}\, is nothing but the order of a
representative vector field $X$ , i.e. the degree of the first non-zero jet of $X$, at the singular point in question.
With this notation, assume that $\tilf$ is obtained by blowing up $\fol$ at a singular point $p$.
Assume also that $S$ is a (formal) separatrix of $\fol$ at $p$ and denoted by $\widetilde{S}$ the transform of $S$ which
yields a (formal) separatrix for $\tilf$ at the point $\widetilde{p}$. Then we have:
\begin{equation}\label{decreasingmultiplicity}
{\rm mult} (\tilf, \widetilde{S})  \leq  {\rm mult} (\fol, S)
\end{equation}
{\it with equality holding if and only if the order of $\fol$ at $p$ equals~$1$}.
The same formula holds for blowing ups centered
at a curve contained in the singular set of $\fol$, up to considering a variant of the notion of
``order of the foliation'' that
is adapted to the center of the blow up, for details see \cite{helenajulio_RMS} or the discussion
at the end of Section~\ref{separatricescod1}.

It is easy to understand the interest of the multiplicity of a foliation $\fol$ along a separatrix from the above perspective:
if the existence of a (formal) separatrix $S$ is ensured, then its multiplicity will drop strictly providing that the
order of $\fol$ at the singular point in question is greater than or equal to~$2$. Moreover, once this decreasing
sequence stabilizes, then the corresponding singular point is either elementary or {\it nilpotent}. From this it also follows
that it is useful to understand {\it persistent nilpotent singular points}\, in order to establish resolution
theorems for foliations.

Clearly, in dimension~$2$, persistent nilpotent singular points do not exist as follows from Seidenberg theorem. In dimension~$3$,
their existence is established by the above discussed examples due to Sancho and Sanz. A characterization of these
points in dimension~$3$ in terms of normal forms can be formulated as follows.

\begin{theorem}
\label{Microlocalversion_TheoremA}
\noindent {\rm \cite{helenajulio_RMS}}\hspace{0.1cm}
Assume that $\fol$ cannot be resolved by a finite sequence of standard blow-ups centered at singular sets. Then there exists
a sequence of one-point blow ups (centered at singular points) leading to a foliation $\fol'$ with a singular point $p$ around
which $\fol$ is given by a vector field of the form
$$
(y + zf(x,y,z)) \frac{\partial}{\partial x} + zg(x,y,z) \frac{\partial}{\partial y} + z^n \frac{\partial}{\partial z}
$$
for some $n \geq 2$ and holomorphic functions $f$ and $g$ of order at least~$1$ with $\partial g / \partial x (0,0,0) \ne 0$.
Furthermore we have:
\begin{itemize}
  \item[(1)] The resulting foliation $\fol'$ admits a formal separatrix at $p$ which is tangent to the $z$-axis;

  \item[(2)] The exceptional divisor is locally contained in the plane $\{ z=0\}$.
\end{itemize}
\end{theorem}

Theorem~\ref{Microlocalversion_TheoremA} deserves a couple of comments as it has a natural analogue in \cite{danielMcquillan},
namely:

\begin{itemize}
		
	\item In \cite{danielMcquillan}, the authors obtain an alternative characterization of persistent nilpotent
	singularities which are
	presented as singular points arising from elementary ones by means of a $\Z /2\Z$-orbifold singularity, see
	Section~\ref{subsec:daniel} for more details. It is relatively
	straightforward to establish the equivalence between their
	characterization and the normal forms provided by Theorem~\ref{Microlocalversion_TheoremA}.

    \item As in \cite{danielMcquillan}, an immediate consequence of the normal forms in Theorem~\ref{Microlocalversion_TheoremA} is
    that every persistent nilpotent singular point can immediately be turned into elementary ones by means of a single
    blow-up of weight~$2$, see \cite{helenajulio_RMS}.
\end{itemize}

In closing this section, let us point out the family of vector fields described in Theorem~\ref{Microlocalversion_TheoremA}
is a genuine extension of the Sancho-Sanz family, albeit one naturally obtained by following their construction.
Indeed, for persistent nilpotent singular points, the multiplicity of the foliation along the corresponding (formal)
separatrix is fully invariant under blow ups whose centers are contained in singular sets, cf, Formula~\ref{decreasingmultiplicity}.
In the Sancho-Sanz family, all multiplicities are equal to~$2$ so that for $n \geq 3$, Theorem~\ref{Microlocalversion_TheoremA}
yields examples that cannot be turned in Sancho-Sanz examples by means of successive blow ups (and conversely).
For example, vector fields in the family
$$
X_{\lambda} = (y - \lambda z) \frac{\partial}{\partial x} + zx \frac{\partial}{\partial y} +
z^3 \frac{\partial}{\partial z} \, ,
$$
with $\lambda \neq 0$, yield foliations $\fol_{\lambda}$ with persistent nilpotent singularities arising from a
(strictly) formal separatrix $S_{\lambda}$.  The multiplicity of $\fol_{\lambda}$ along $S_{\lambda}$ being equal to~$3$.


\subsection{The desingularization theorem of McQuillan-Panazzolo}\label{subsec:daniel}

The purpose of this paragraph is to explain in detail the desingularization theorem proved in \cite{danielMcquillan}.
As previously mentioned, McQuillan and Panazzolo work from the start in the category of weighted blow-ups, thus
not limiting themselves to stardard ones.
Unlike standard blow-ups, that keep the smooth nature of the space, weighted blow-ups lead to singular ambient spaces.
Yet, the singularities in question are of orbifold-type and hence of a rather simple nature. Whereas singular,
it should be pointed, that the ambient space
obtained after a sequence of finitely many weighted blow ups still is
{\it birationally equivalent to the initial one}. In particular, foliations can be transformed
without any restrictions under weighted blow ups to yield new birational models for them.

Keeping in mind the issues pointed out above, let us summarize the contents of \cite{danielMcquillan}.
The paper \cite{danielMcquillan} is essentially divided into two parts. Its first part is devoted to
prove that the algorithm of \cite{daniel2} - leading to a resolution of singularities by means of a
sequence of weighted blow-ups for real analytic foliations on
$(\R^3,0)$ - applies equally well in the general case of holomorphic foliations on $(\C^3,0)$. The algorithm
in question thus provides a birational model for the foliation on a space possessing orbifold-type singular
points. Note that, since we are dealing with (singular)
foliations on spaces with singular points of orbifold type, a word is needed about the meaning
of ``elementary singular points''. In this regard, the singular point of the foliation is said to
be {\it elementary} if it is given by an elementary singular point in a {\it orbifold coordinate}\, for the space.
In particular, there are an open set $U \subset \C^3$ and finitely ramified map from $U$ to a neighborhood of
the orbifold singular point such that when the foliation is pulled-back to the open set $U \subset \C^3$ only
elementary singular points are obtained.

In the second part of \cite{danielMcquillan}, the authors consider the problem of resolving the
singular points of the ambient space while keeping the singular points of the foliation elementary.
They prove that a resolution for such singularities exists except when
the singular point correspond to a $\Z /2\Z$-orbifold. In other words, they have shown that, given a
foliation $\fol$, it is always
possible to obtain a birational model for $\fol$ possessing only $\Z /2\Z$-orbifold singular points
and where all the singular points of the foliation in question are elementary. In turn, the singularities
associated with $\Z /2\Z$-orbifolds that appear in the end of
the previous construction can be turned into elementary singularities by means of a single blow-up
of weight~$2$. These singularities associated with $\Z /2\Z$-orbifolds actually correspond to the
previously described persistent nilpotent singular points.

The result in \cite{danielMcquillan} can thus be stated as follows:

\begin{theorem}\label{danielJDG}
\noindent {\rm \cite{danielMcquillan}}\hspace{0.1cm} Let $\fol$ be a singular holomorphic foliation on $(\C^3,0)$. There is a
sequence of weighted blow-ups
\begin{equation}
\fol_0 \stackrel{\pi_1}\longleftarrow \fol_1 \stackrel{\pi_2}\longleftarrow \cdots
\stackrel{\pi_l}\longleftarrow \fol_l \label{weightedblowups}
\end{equation}
satisfying the following conditions:
\begin{itemize}
  \item[(i)] The center of each weighted-blow up is strictly invariant with respect to the quasi-homogeneous filtration in question.

  \item[(ii)] The ambient space is an analytic space of dimension~$3$ whose singular points are $\Z /2\Z$-orbifold type
  and the total blow-up map $\pi_1 \circ \cdots \circ \pi_l$ is birational.

  \item[(iii)] The singular points of $\fol_l$ are elementary in orbifold coordinates.
\end{itemize}
\end{theorem}

Let us close this paragraph with a comment concerning item~(i) of Theorem~\ref{danielJDG}. Note that this item is not
emphasized in \cite{danielMcquillan} though it is a characteristic property of Panazzolo's algorithm in \cite{daniel2}.
Whereas, as far as foliations are concerned, this is a minor issue - as it would also be the case of blow ups centered
away from the singular locus (whether or not the blow ups are weighted) - the issue becomes relevant when our main interest
lies in vector fields, rather than foliations, see Section~\ref{transformingvectorfields}


\subsection{Resolution following \cite{helenajulio_RMS}}\label{subsec:nos}

In \cite{helenajulio_RMS}, we also establish the existence of a birational model for $\fol$ where all singularities of $\fol$ are elementary except
for finitely many ones that can be turned into elementary singular points by means of a single blow-up of weight 2. To be more precise,
our resolution result for foliations can be stated as follows.

\begin{theorem}\label{teo:A}
\noindent {\rm \cite{helenajulio_RMS}}\hspace{0.1cm}
Let $\fol$ denote a singular holomorphic foliation defined on a neighborhood of $(0,0,0) \in \C^3$. Then there
exists a finite sequence of blow-up maps along with transformed foliations
\begin{equation}
\fol = \fol_0 \stackrel{\pi_1} \longleftarrow \fol_1 \stackrel{\pi_2} \longleftarrow \cdots
\stackrel{\pi_l}\longleftarrow \fol_n \label{blowup-resolution_globaldesingularization}
\end{equation}
satisfying all of the following conditions:
\begin{itemize}
  \item[(1)] The center of the blow-up map $\pi_i$ is (smooth and) contained in the singular set of $\fol_{i-1}$, $i=1, \ldots , n$.

  \item[(2)] The singularities of $\fol_n$ are either elementary or persistently nilpotent.

  \item[(3)] The number of persistently nilpotent singularities of $\fol_n$ is finite and each of them can be turned into
  elementary singular points by performing a single weighted blow-up of weight~$2$.
\end{itemize}
\end{theorem}

The proof of this theorem has essentially two main ideas. The first one concerns a (personal) comment by F. Cano claiming
that ``if a foliation cannot be resolved by standard blow ups, then there must exist a formal separatrix
giving rise to a sequence of infinitely near singular points that never becomes elementary''. This assertion harks back
to his earlier works on resolutions of $1$-dimensional foliations \cite{canoLNM} and some important results in this direction
can also be found in \cite{C-R-S}. To provide a complete proof of Cano's assertion was therefore a crucial
point in the proof of Theorem~\ref{teo:A} and the corresponding result is the content of
Proposition~4 in \cite{helenajulio_RMS}. Interestingly enough, the argument
provided in \cite{helenajulio_RMS} is rather different from the one envisaged by F. Cano.

With Proposition~4 of \cite{helenajulio_RMS} in place, the main idea to derive Theorem~\ref{teo:A} is to argue from the
notion of multiplicity of a foliation along a separatrix, as defined in Subsection~\ref{subsec:pns}. The sequence
formed by a separatrix and its transforms is decreasing so that it stabilizes after finitely steps. When the sequence becomes
stable, the order of the singular point of the foliation must be~$1$. Thus either the singularity has become elementary
or we can resort to Theorem~\ref{Microlocalversion_TheoremA} to characterize it as a persistent nilpotent singularity,
which is necessarily isolated among other possible persistent nilpotent singular points. Therefore this yields a
{\it local uniformization theorem}\, in the sense of Zariski. At this point, O. Piltant ``gluing theorem'' \cite{Piltant} allows one
to conclude Theorem~\ref{teo:A}.

Recalling that persistent nilpotent singular points are in correspondence with $\Z /2\Z$-orbifold type singular points,
the differences between Theorem~\ref{danielJDG} and Theorem~\ref{teo:A} are down to the way the corresponding birational
models are constructed. Unlike Panazzolo \cite{daniel2}, our proof of Theorem~\ref{teo:A} does not provide any effective algorithm
to resolve singularities. In some problems, however, it might simplify discussions/arguments by sticking to a single type of blow up,
the standard one, provided that the problem in question requires only a theorem asserting the existence of a resolution,
as opposed to an effective manner to obtain the resolution in question.


\subsection{A final comment on transforming vector fields}\label{transformingvectorfields}

In close this section, let us point out a virtue of standard blow ups, as used as in Theorem~\ref{teo:A}, that is also present
in Theorem~\ref{danielJDG} thanks to item~(i) in the corresponding statement. This concerns vector fields as opposed to
$1$-dimensional foliations.

In fact, it is not a foliation but rather some holomorphic
vector field that is the object of primary interestin many problems and applications of singularity theory.
Examples of this situation are provided in
Section~\ref{separatricescod1} and throughout Section~\ref{semicompletevectorfields}. Naturally a vector field $X$
gives rise to an $1$-dimensional foliation $\fol$ of which a birational model whose all singular points are elementary may
be useful. Nonetheless, if the vector field $X$ is the object of primary interest, then the transforms of $X$ have
to be considered as well. At this point, the difference between vector fields and $1$-dimensional foliations
is summarized by the following self-evident statement: the transform of a $1$-dimensional holomorphic foliation under a
rational map is another $1$-dimensional holomorphic foliation, however, the transform of a holomorphic vector field
under a rational map {\it is, in general, a meromorphic vector field}.

When applying resolution theorems for $1$-dimensional foliations to the study of vector fields
it is therefore relevant to seek to retain the ``good'' analytic properties of them as much as possible (again concrete examples
are provided in Sections~\ref{Allsortsofseparatrices} and~\ref{semicompletevectorfields}). For example,
if we start with a holomorphic vector field $X$, we might hope
that its transform at the end of a resolution procedure still is a holomorphic vector field. In this context, we have:

\vspace{0.1cm}

\noindent {\it Claim}. The transform of a holomorphic vector field under a resolution procedure as in
Theorem~\ref{danielJDG} or in Theorem~\ref{teo:A} still is a holomorphic vector field.

\vspace{0.1cm}

The claim is clear in the case of Theorem~\ref{teo:A} as it is a basic fact that the blow-up of a holomorphic
vector field centered at its singular locus is again holomorphic. In fact, for the blown-up vector field to be holomorphic again
it suffices the center of the blow-up to be invariant by the initial vector field.

In the case of Theorem~\ref{danielJDG} this is not immediate as a weighted blow-up may turn a holomorphic vector field
into a meromorphic one even if its center is contained in the singular set of the initial vector field. This is where
{\it the condition of having centers that are called strictly invariant with
respect to the quasi-homogeneous filtration in question, as used in \cite{daniel2} and reproduced in the first part of \cite{danielMcquillan},
comes into play (see item~(i) in Theorem~\ref{danielJDG})}. This condition, if slightly technical, ensures
that the transform of holomorphic vector fields remains holomorphic.


\section{Invariant analytic sets}\label{Allsortsofseparatrices}

The problem of existence of {\it invariant manifolds}\, has always been a central theme in the theory of dynamical systems.
Among the many reasons for this, there is the simple fact that these invariant manifolds
usually provide reductions on the dimension of
the corresponding phase-space. For example, in the general theory of hyperbolic systems, the so-called stable manifolds
are examples of invariant manifolds and, in fact, their existence form a cornerstone of the hyperbolic theory. The existence
of stable manifolds for hyperbolic singular points is a consequence of the general theory and ensured by the well-known Stable
Manifold Theorem. However, whether or not ``stable'', invariant manifolds may fail to exist if
the singular point is no longer hyperbolic. The simplest example is provided by the vector field
\[
X = y \frac{\partial}{\partial x} - x \frac{\partial}{\partial y} \,
\]
all of whose integral curves are circles around the origin of $\R^2$. Clearly there is no
invariant manifold in this case.

The general problem of existence of invariant manifolds may also be considered in the context
of holomorphic dynamics. In this
case, we look for invariant complex-analytic objects, which is a much stronger regularity condition.
We allow, however, these objects to be singular in the sense of analytic sets. In other words,
we look for {\it invariant varieties}, as opposed to actual {\it manifolds}. In the sequel, the word
``manifold'' will be saved for smooth objects.

As mentioned in section~\ref{basics}, Briot and Bouquet were the first to consider the problem of existence of separatrices
for holomorphic vector fields defined on a neighborhood of the origin of $\C^2$ in \cite{briotbouquet}. However,
they were not able to establish the existence of
separatrices for all holomorphic vector fields on $(\C^2,0)$. This question was settled
only much later  by Camacho and Sad in their remarkable paper \cite{camachosad} where the following is proved:

\begin{theorem}\cite{camachosad}\label{teo_CS}
Let $\fol$ be a singular holomorphic foliation defined on a neighborhood of the origin of $\C^2$. Then there exists an analytic
invariant curve passing through $(0,0)$ and invariant by $\fol$.
\end{theorem}

Theorem~\ref{teo_CS} is well worth a few additional comments, namely:

\begin{itemize}
\item It is somehow surprising that separatrices for holomorphic vector fields on $(\C^2,0)$ always exist despite the much stronger
regularity condition for the invariant curve. For example, for the holomorphic vector field $y \frac{\partial}{\partial x}
- x \frac{\partial}{\partial y}$ defined on $(\C^2,0)$, the separatrices are given by the two complex lines $y = \pm ix$
and hence are totally contained in the non-real part of $\C^2$ (bar the singular point itself).

\item However, as mentioned, we do not require the separatrices to be smooth invariant curves
otherwise no general existence statement would hold. Indeed, as a simple example,
consider the holomorphic vector field $2y \partial /\partial x + x^3 \partial /\partial y$ on $(\C^2,0)$. Since this vector field admits
$f(x,y) = x^3 - y^2$ as first integral, it immediately follows that the only separatrix of $X$ is the cusp of equation $\{x^3 - y^2
= 0\}$, which is clearly not smooth at the origin. Fortunately, allowing separatrices to be singular is not a problem,
since Hironaka's theorem can always be used to desingularize them.

\item Also it is important to emphasize that Theorem~\ref{teo_CS} applies only to foliations defined on
smooth ambients. In fact, if germs of foliations defined on singular surfaces are considered, then
separatrices may fail to exist as shown by Camacho
in~\cite{C}.
\end{itemize}

The existence of separatrices is, however, no longer a general phenomenon in dimension~$3$, regardless of the dimension of the
foliation. As already said, a first example of codimension~$1$ foliation on $(\C^3,0)$ without separatrices
was provided by Jouanolou in \cite{jouanolou}. Jouanolou's example essentially hinging from the core dynamics
of (dicritical) foliations on $(\C^3,0)$, the same idea enables us to construct plenty of additional examples of
codimension~$1$ foliations without separatrices (cf. Section~\ref{subsec:Jouanolou} or the summary below).

Concerning $1$-dimensional foliations, examples of foliations without separatrices in dimension~$3$ were found by
Gomez-Mont and Luengo, \cite{GM-L}. Their work will be discussed in Section~\ref{separatrices_1dimension}.
For the time being, we will focus on the problem of invariant manifolds for codimension~$1$ foliations.


\subsection{Separatrices for codimension~$1$ foliations induced by pairs of commuting vector fields}\label{separatricescod1}

Let us begin by recalling/summarizing the discussion in
Section~\ref{subsec:Jouanolou} where it was shown how Jouanolou's method can be used to produce
many examples of codimension~$1$ foliations without
separatrix on $(\C^3,0)$. This is as follows.
\begin{itemize}
\item[(i)] Every homogeneous polynomial vector field $X$ on $\C^3$ that is not a multiple of the Radial vector field
induces a foliation on $\CP^2$ corresponding to the so-called core foliation associated with $X$.
Conversely, given a foliation
on $\CP^2$, there exists a homogeneous polynomial vector field on $\C^3$
whose core foliation is the given one.

\item[(ii)] Let $X$ be a homogeneous vector fields distinct from a multiple of the Radial vector field $R$. There follows
from the Euler relation that $X$ and $R$ generates a Lie algebra isomorphic to the Lie algebra of the affine group. In particular,
the distribution generated by $X$ and $R$ can be integrated to yield a dicritical codimension~$1$ foliation $\cald$. Furthermore,
the core foliation associated
with $\cald$ coincides with the core foliation associated with $X$.

\item[(iii)] Finally, for every fixed degree, Theorem~\ref{FrankandJulio} ensures the existence of a (non-empty) open set of foliations
on $\CP^2$ such that all leaves of each foliation $\fol$ in this set are dense. In particular,
no foliation in this set admits
algebraic invariant curves. The codimension~$1$ foliations generated by $R$ and by the homogeneous vector field having
one such foliation as core foliation has no seraparatrix.
\end{itemize}

In view of what precedes, it is natural to wonder
if all example of codimension~$1$ foliations without
separatrices are among the dicritical ones. In ambient spaces of dimension~$3$ this, in fact, holds as proved by Cano and Cerveau
in \cite{CanoCerveau}.
Their result can be stated as follows.

\begin{theorem}\label{CanoCerveauSep}\cite{CanoCerveau}
Let $\cald$ be a germ of a holomorphic singular codimension~$1$ foliation on $(\C^3,0)$.
If $\cald$ is not dicritical, then it admits a separatrix.
\end{theorem}

The proof of the previous result relies heavily on a resolution theorem for non-dicritical codimension~$1$ foliations
obtained by the authors in the same paper.
Note, however, that the non-dicritical assumption, implies that the transforms of the initial codimension~$1$ foliation
leave every irreducible component of the exceptional divisor
invariant. In other words, away from singular points,
every irreducible component of the exceptional divisor is a leaf of the foliations in question. This rules out
the existence of meaningful core dynamics and making the problem very much comparable to the $2$-dimensional
situation handled by Camacho and Sad in \cite{camachosad} which has a more geometric nature.

In a different direction, experts including F. Cano, D. Cerveau, and L. Stolovitch have since long wondered
what would be the ``correct generalization'' of Camacho-Sad theorem for $(\C^3, 0)$, already at level of
codimension~$1$ foliations. In particular, the idea that a codimension~$1$ foliation spanned by a pair of
commuting vector fields (not everywhere parallel) might necessarily admit separatrices was advanced.
The question is settled by the theorem below which confirms their intuition.

\begin{theorem}\cite{RR_separatrix}\label{teo_RR_separatrix}
Consider holomorphic vector fields $X, \, Y$ defined on a neighborhood of the origin of $\C^3$. Suppose that they
commute and are linearly independent at generic points (so that they span a codimension~$1$ foliation denoted by
$\cald$). Then $\cald$ possesses a separatrix.
\end{theorem}

The remainder of this paragraph is devoted to single out a few issues involved in the proof of
Theorem~\ref{teo_RR_separatrix}. This illustrates several points made in the preceding sections,
including the usefulness of ``taming'' core dynamics (and how ``symmetries'' may be exploited to this effect)
and the role of resolutions theorems. Concerning the latter, the argument will also highlight the
the importance of having actual vector fields - rather than mere
foliations - being ``nicely'' transformed during the resolution procedure.

The first ingredient in the proof of Theorem~\ref{teo_RR_separatrix} is therefore
a general resolution of singularities theorem
for codimension~$1$ foliations in dimension~$3$. Compared to Theorem~\ref{CanoCerveauSep}, the main result in
\cite{CanoCerveau} is arguably a theorem of reduction of singularities for the foliations in question {\it under the
additional condition}\, that the foliation should be non-dicritical. Fortunately, Cano has obtained in \cite{cano}
a general resolution theorem for codimension~$1$ foliations on $(\C^3,0)$ which applies equally
well to dicritical foliations.

Armed with Cano's theorem \cite{cano}, we see that the basic obstacle for the existence of separatrices lies in
the core dynamics by means of the phenomenon already pointed out in Jouanolou examples, cf. Section~\ref{core_dynamics}.
The central point in the
proof of Theorem~\ref{teo_RR_separatrix} is therefore to ``tame'' the core dynamics arising from dicritical
divisors the resolution procedure applied to $\cald$ will have ``plenty of algebraic curves''. In the sequel,
we shall indicate some simple ideas used to show that the mentioned core dynamics cannot be ``too wild''.

Let us consider the simplest case where we want to blow-up the origin (a degenerate singular point of $\cald$).
The first lemma shows that at least one between the vector fields $X$ and $Y$ have to induce a foliation on
the resulting exceptional divisor, unless we have a truly very special situation that is essentially
``linear'' (and hence easy to handle).

Recalling that $\cald$ is spanned by the commuting vector fields $X$ and $Y$, let $\fol_X$ (resp. $\fol_Y$) denote
the $1$-dimensional singular foliation associated with $X$ (resp. $Y$).

\begin{lemma}\label{nondicriticalvectorfields}
	Assume that the first jet of both $X$ and $Y$ at the origin are equal to zero. Then none of the foliations
	$\fol_X$ or $\fol_Y$ is dicritical for the blow-up $\pi$ of $\C^3$ at the origin.
\end{lemma}

\begin{proof}
Denote by $\widetilde{X}$ and $\widetilde{Y}$ the blow-ups of $X$ and $Y$ at the origin. Similarly,
$\tilf_x$ and $\tilf_Y$ will stand for the blow ups of the foliations $\fol_X$ and $\fol_Y$.
Since the vector fields $X$ and $Y$
have zero linear part at the origin, there follows that both $\widetilde{X}$ and $\widetilde{Y}$ vanish identically
over the exceptional divisor $\pi^{-1} (0) \simeq \CP^2$. Now assume that, say, X is dicritical for $\pi$. Then
the leaf of $\tilf_X$ is regular and transverse to $\pi^{-1} (0)$ at generic points of $\pi^{-1} (0)$.
Therefore, around one such point, we can choose
local coordinates $(u,v,w)$ such that $\{ u=0\} \subset \pi^{-1} (0)$ and where
of $\tilf_X$ is represented by the (regular) vector field $\partial /\partial u$. In particular the blow-up
$\widetilde{X}$ takes on the form $f(u,v,w) \partial /\partial u$ where $f$ is a holomorphic function (divisible by $u$).
In these coordinates, let the blow-up $\widetilde{Y}$ be given by $\widetilde{Y} = f_1 \partial /\partial u
+ f_2 \partial /\partial v + f_3 \partial /\partial w$. Since $[\widetilde{X}, \widetilde{Y}] =0$, there follows
that $f_2$ and $f_3$ do not depend on the variable $u$. However, these functions must vanish identically for $u=0$
since $\widetilde{Y}$ vanishes identically over $\pi^{-1} (0)$ (locally given by $\{ u=0\}$). Thus they must
vanish identically over an open set and this contradicts the fact that $X$ and $Y$ span a codimension~$1$ foliation.
\end{proof}

\begin{obs}
	{\rm The argument above shows the importance of transforming vector fields, as opposed to foliations, in certain
		cases. In fact, the proof of Lemma~\ref{nondicriticalvectorfields} hinges from the fact that the transform
		of the vector field $Y$ vanishes identically over the exceptional divisor $\pi^{-1} (0)$ - something that does not make
		sense for a foliation since the singular set of the latter has codimension at least~$2$.
		
		Along similar lines, to ensure that the transformed vector field $\widetilde{Y}$ vanishes identically
		over $\pi^{-1} (0)$, the fact that the origin (center of the blow up) is contained in the
		singular set of $\fol_X$ (or more generally, the center of the blow up is invariant under the foliation)
		was implicitly used. This is in line with the discussion in Section~\ref{transformingvectorfields}.
		It is often important that the transformed vector field retains its holomorphic character. In addition,
		in quite a few cases, it is also important that the zero-divisor of the transformed vector field
		contains all components of the exceptional divisor arising from the resolution procedure.
		
		Plenty of additional examples of this issue can be found in the theory of semicomplete vector fields,
		see for example \cite{Ghys-R}, \cite{guillotadvances}, or \cite{Guillot-Rebelo}.}
\end{obs}

Lemma~\ref{nondicriticalvectorfields} shows that both $\fol_X$ or $\fol_Y$ must induce a foliation on
$\CP^2$. If the foliations induced are different, then it follows from the discussion in Section~\ref{core_dynamics}
that $\pi^{-1} (0)$ is invariant by $\cald$. Hence we can assume that they do coincide. Recalling that the {\it order}\,
of a vector field at a singular point $p$ is nothing but the degree of the first non-zero homogeneous component
of its Taylor series based at the point in question, the preceding implies:

\begin{lemma}\label{differentorders}
The orders at the origin of $X$ and $Y$ can be assumed to be different.
\end{lemma}

\begin{proof}
	Assume that $X$ and $Y$ have the same order at the origin. Because they induce the same foliation on $\CP^2$,
	they will differ by a multiple of the radial vector field (up to multiplying, say X, by a non-zero constant).
Hence, by considering $Z = X-Y$, there follows that the foliation $\cald$ is still spanned by $X$ and $Z$ and the
we still have $[X,Z]=0$. This is however impossible since $Z$ is clearly dicritical so that a contradiction with
Lemma~\ref{nondicriticalvectorfields} arises at once.
\end{proof}

Denote by $X^H$ (resp. $Y^H$) the first non-zero homogeneous component of $X$ (resp. $Y$) at the origin.
Owing to the above lemma, we can assume that the degree of $Y^H$ is {\it strictly greater}\, than the degree of $X^H$.
The preceding implies that the core dynamics of either $X^H$, $Y^H$ coincides with the core dynamics of the
dicritical foliation $\cald$. Now the following proposition provides some serious control on the core dynamics
in question and, along with its analogue for blow ups centered at curves, constitutes a fundamental starting point
of the discussion conducted in \cite{RR_separatrix}.

\begin{prop}\label{existenceoffirstintegrals}
	The vector field $X^H$ admits a non-constant meromorphic/holomorphic first integral.
\end{prop}

\begin{proof}
	Owing again to the discussion in Section~\ref{core_dynamics}, the dicritical nature of $\cald$
	ensures the existence of holomorphic functions $f$ and $g$ such that
	\begin{equation}
	fX + gY = Z \, , \label{fgh1}
	\end{equation}
	with $Z$ being a holomorphic vector field whose first non-zero homogeneous component at the origin
	is a multiple of the radial vector field~$R$. Denoting by $f^H,g^H$ the first non-zero homogeneous
	components of $f,g$, there follows from the preceding that $f^H X^H$ and $g^H Y^H$ must have the
	same degree. Furthermore, we have a homogeneous equation
	\begin{equation}
	f^H X^H + g^H Y^H = h^H R \label{fghH}
	\end{equation}
	where $h^H$ is a homogeneous polynomial - possibly identically zero. In the sequel we assume that $h^H$ does
	not vanish identically since it is easy to adapt the discussion below to cover this case as well.

Because $X,Y$ commute, so do $X^H ,Y^H$. Thus we have
\begin{eqnarray*}
	[ X^H , Y^H ] &=& \left[ X^H , \frac{h^H}{g^H} R - \frac{f^H}{g^H} X^H \right]\\
	&=&  \left[ X^H.\left(\frac{h^H }{g^H}\right) \right] R - \frac{h^H}{g^H} [R,X^H] - \left[X^H.\left(\frac{f^H}{g^H}\right)\right] X^H\\
	&=& \left[ X^H.\left(\frac{h^H}{g^H}\right) \right] R - \left[(d-1) \frac{h^H}{g^H} - X^H.\left(\frac{f^H}{g^H}\right)\right] X^H\\
	&=& 0
\end{eqnarray*}
where $d$ stands for the degree of $X^H$. In particular
\[
\left[ X^H.\left(\frac{h^H}{g^H}\right) \right]  R = \left[(d-1) \frac{h^H}{g^H} - X^H.\left(\frac{f^H}{g^H}\right)\right] X^H \, .
\]
The expression between brackets on the left hand side (i.e. the
expression multiplying $R$) must vanish identically for otherwise $X^H$
would be a multiple of the Radial vector field~$R$. It then follows that
$$
X^H.\left(\frac{h^H}{g^H}\right) =0 \, .
$$
In other words,  $h^H/g^H$ is a meromorphic (possibly holomorphic) first integral for $X^H$.

It only remains to prove that $h^H/g^H$ is not constant. However, if this function is constant
(different from zero since $h^K$ does not vanish identically), then
we can assume
$h^H/g^H = 1$ without loss of generality. Hence dividing
(\ref{fghH}) by $g^H$, it would follow
\[
\frac{f^H}{g^H} X^H + Y^H = R .
\]
This last equation is however impossible since $Y^H$ has degree at least~$2$ and the expression
$f^HX^H/g^H$ is homogeneous. Therefore $h^H /g^H$ cannot be constant. Since the argument
is symmetric in the vector fields $X, Y$, the last assertion completes our proof.
\end{proof}

The key to prove Theorem~\ref{teo_RR_separatrix} is to observe that the core dynamics of dicritical components
of a foliation like $\cald$ must leave invariant certain algebraic curves. Clearly
Proposition~\ref{existenceoffirstintegrals} along with some refinements play a role in this proof.

However, it is also clear that the discussion leading to Proposition~\ref{existenceoffirstintegrals} depends
heavily on the vanishing assumption for the first jet of $X$, $Y$ at the origin. This issue requires to consider
separately some special situations that are referred to as ``linear foliations'' in the sense that there is
a {\it non-zero first jet involved}. Not surprisingly, ``linear foliations'' can be dealt with through rather
direct methods.

Finally, as it is inevitable in dimension~$3$, every desingularization procedure requires {\it two types of blow-ups}:
beyond blow-ups centered at points, blow-ups centered at curves are needed as well. In particular,
another basic ingredient in the proof of Theorem~\ref{teo_RR_separatrix} will be analogues, both in
``linear'' and ``non-linear'' settings, of the previous results. This issue has already appeared
in Section~\ref{resolution_blowups} (cf. the discussion about Equation~\ref{decreasingmultiplicity})
and can easily lead to misunderstandings so that it seems convenient to close this paragraph
by carefully explaining the appropriate formulations.

The following example was pointed out to us by D. Cerveau. It helps to explain the notion of
``zero first jet'' in the case of blow-ups centered at smooth curves. The example also highlights difficulties
related to the existence of first integrals (cf. Proposition~\ref{existenceoffirstintegrals} and a few
other intermediary results used in \cite{RR_separatrix} and not explicitly mentioned here).

\begin{ex}\label{fortimebeingunique}
Consider the pair of vector fields $X, Y$ given by
$$
X = zy \frac{\partial}{\partial y} + z^2 \frac{\partial}{\partial z} \; \qquad {\rm and} \; \qquad Y = x^2 \frac{\partial}{\partial x} + axy \frac{\partial}{\partial y} \, \,
$$
which are quadratic at the origin. Note that
these two vector fields commute so that they span a codimension~$1$ foliation denoted by $\cald$.

The axis $\{ y=z=0\}$ is invariant by both $X$ and $Y$. This axis is also contained in the singular set of $\cald$.
Let us then consider the blow-up of $\C^3$ centered at the axis $\{ y=z=0\}$ along with the
corresponding transforms of $\cald$, $X$, and $Y$.
It is immediate to check $\cald$ is dicritical for the blow up in question. Similarly the foliation $\fol_X$
associated with $X$ is also dicritical for this blow up which might lead to some confusion with
Lemma~\ref{nondicriticalvectorfields}.

The explanation for this example lies in the fact that the vector field $Y$ is regular (non-zero) at generic points of
the axis $\{ y=z=0\}$. Similarly, its transform under the previous blow up is regular at generic points
of the exceptional divisor. In other words, this case must be considered as a ``linear one'' and the order
of $Y$ with respect to this blow up must, indeed, be equal to {\it zero}. An adequate definition of the
order of a vector field with respect to the center of a blow up is included below.
\end{ex}

Let us then provide an accurate definition of {\it order}\, of a vector field when a curve, as opposed to a single
point, is blown-up. To explain the idea, consider first a holomorphic vector field $X$
with a singular point at the origin along with the corresponding Taylor series. The
order of $X$ at the origin is said to be the degree of the first non-zero homogeneous component of its Taylor expansion. This can
also be viewed as the integer $d$ for which the limit
\[
\lim_{\dl \rightarrow 0} \frac{1}{\dl^{d-1}} \Gamma_{\dl}^{\ast} X
\]
yields a (non identically zero) holomorphic vector field.
Here $\Gamma_{\dl}^{\ast} X$ stands for the pull-back of $X$ by the homothety
$\Gamma_{\dl} : (x,y,z) \mapsto (\dl x, \dl y, \dl z)$. Note that the limit above corresponds to the first non-zero
homogeneous component of the Taylor's expansion of $X$ at the origin. Next, assume now that $C = \{y=z=0\}$ is contained
in the singular set of $X$ so that the blow-up centered along this curve of singular points will be considered.
The order of $X$ with respect to $C$ is defined as the integer $d$ for which
\[
\lim_{\dl \rightarrow 0} \frac{1}{\dl^{d-1}} \Lambda_{\dl}^{\ast} X
\]
is a (non identically zero) holomorphic vector field, where $\Lambda_{\dl}^{\ast} X$ denotes the pull-back of
$X$ by the homothety $\Lambda_{\dl} : (x, y, z) \mapsto (x, \dl y, \dl z)$. The limit above, for the appropriate
choice of $d$, is said to be the first
non-zero homogeneous component of $X$ with respect to the variables $x,y$. In general, the cases in
\cite{RR_separatrix} that are called {\it linear}\, are those cases in which the vector field has order~$1$ or
zero, with respect to the center of the blow up in question. In particular, with the above definition,
it can immediately be checked
that the vector field $Y$ of Example~\ref{fortimebeingunique} has order {\it zero} with respect to $C = \{y=z=0\}$,
although its order at the origin is~$2$.


\subsection{Separatrices for foliations of dimension~$1$}\label{separatrices_1dimension}

This paragraph is devoted to discussing in detail the problem about existence of separatrices
{\it for $1$-dimensional foliations}. Contrasting with the case of codimension~$1$ foliations,
it will soon be seen that the influence of core dynamics in the existence of these separatrices
is rather limited.
In fact, the existence of separatrices for foliations of dimension~$1$ is an phenomenon having, in a suitable sense,
a {\it very local nature}: it essentially hinges from two basic ingredients, namely:
\begin{itemize}
	\item The analysis of simple singularities which is basically conducted by direct methods involving
	normal forms and divergent series.
	
	\item Geometric considerations involving the relative positions of the simple singularities in question.
\end{itemize}
In this regard, and provided that a convenient resolution of singularities theorem is available,
the problem somehow retains the same nature regardless of the dimension. More precisely, the difficulties arising
from increasing the dimension stem either from the evident fact that simple singularities are not always easy to
describe (e.g. saddle-nodes of high codimension) and from the fact that the number of possible arrangements
of their relative positions increase as well.

As previously said, after Camacho-Sad theorem in \cite{camachosad}
establishing the existence of separatrices for every foliation on $(\C^2,0)$, Gomez-Mont
and Luengo found a foliation on $(\C^3,0)$ that admits no separatrix.
Let us begin by providing an outline of their construction.

\subsection{On Gomez-Mont and Luengo counterexample}
Their example of foliation without separatrix on $(\C^3,0)$ relies on two simple remarks.
Consider a foliation $\fol$ on $(\C^3,0)$ given by a holomorphic vector field satisfying the following conditions
\begin{itemize}
\item[(1)] The origin $(0,0,0) \in \C^3$ is an isolated singularity of $X$

\item[(2)] $J^1 X (0,0,0) = 0$ but $J^2 X (0,0,0) \ne 0$, where $J^k X (0,0,0)$ stands for the jet of order $k$ of $X$ at the origin
($k=1,2$).

\item[(3)] The quadratic part $X^2$ of $X$ at $(0,0,0)$ is a vector field whose singular set has codimension~$2$. Also $X^2$ is not
a multiple of the Radial vector field $x \partial /\partial x + y \partial /\partial y + z \partial /\partial z$.
\end{itemize}
Assume that $\fol$ has a separatrix $C$ and consider the blow-up $\tilf$ of $\fol$ centered at the origin. Denote by $\pi$
the blow-up map so that $\tilf = \pi^{\ast} \fol$ and let $\pi^{-1}(0)$ denote the exceptional divisor isomorphic to $\CP^2$.
Since $X^2$ is not a multiple of the Radial vector field, there follows that $\pi^{-1}(0)$ is invariant by $\tilf$
so that the restriction of $\tilf$ to
$\pi^{-1}(0)$ can be seen as a foliation of degree~$2$ on $\CP^2$ (cf. item (3)).

Because $\pi^{-1}(0)$ is invariant by $\tilf$, the transform $\pi^{-1}(C)$ of the separatrix $C$ can only intersect
$\pi^{-1}(0)$ at singular points of $\tilf$. Furthermore, all of these singular points lie in $\pi^{-1}(0)$
since $X$ has an isolated singularity at the origin. In other words, $\pi^{-1}(C)$ must be a separatrix (not contained in
$\pi^{-1}(0)$) for one of the singular points of $\tilf$.

Now, the second ingredient is as follows: as a foliation of degree~$2$ on $\CP^2$, the restriction
$\tilf|_{\pi^{-1}(0)}$ of $\tilf$ to $\pi^{-1}(0)$ has at most (and generically)
$7$ singular points. Since it is hard to control the position of $7$ points in $\CP^2$, the authors
of~\cite{GM-L} started from a foliation satisfying the following conditions:
\begin{itemize}
\item[(A)] The foliation of degree~$2$ has only~$3$ singular points (we can think of the foliation
as obtained by letting
some of the $7$ singular points of a generic quadratic foliation to ``collide in groups''). Naturally the position
of~$3$ points in $\CP^2$ can easily be controlled.

\item[(B)] Each of the~$3$ singular points will have an eigenvalue equal to zero in the direction transverse
to $\pi^{-1}(0)$. The~$3$ singular points are therefore saddle-node singularities (in dimension~$3$).

\item[(C)] Furthermore, the authors arrange for the saddle-node singularities to have two equal
(and non-zero) eigenvalues tangent to $\pi^{-1}(0)$.
In other words, the singular points in question are (codimension~$1$) resonant saddle-nodes
with weak direction transverse to $\pi^{-1}(0)$.

\item[(D)] As is well known, it is easy to produce examples of codimension~$1$
saddle-nodes all of whose separatrices are included in an invariant ($2$-dimensional) plane tangent to the
directions of the non-zero eigenvalues.

\end{itemize}

The remainder of the proof in \cite{GM-L} consists of showing that it is, indeed, possible to
prescribe a quadratic $X^2$ and a cubic $X^3$
homogeneous components for the vector field $X$ so as to satisfy all of the preceding conditions. In this respect, note
that conditions (A), (B), and (C) depend only on the quadratic part $X^2$. The role played by the
appropriately chosen cubic parte $X^3$ can be summarized as follows.
\begin{itemize}
\item[$\bullet$] it ensures that each of the singular points of $\tilf$ are isolated, hence coinciding with the corresponding singular points of $\tilf|_{\pi^{-1}(0)}$.
Here the reader may note that the homogeneous foliation associated with $X^2$ has singularities all along the fibers
of $\widetilde{\C}^3 \to \widetilde{\pi}^{-1}(0)$ sitting over the singular points of $\tilf|_{\pi^{-1}(0)}$.
A higher order perturbation of $X^2$ is thus needed to provide isolated singular points for the blown-up foliation.

\item[$\bullet$] having ensured the singular points are isolated, the cubic part $X^3$ of $X$ also takes care of condition~(D)
\end{itemize}

As mentioned, the verification that all these conditions are compatible is conducted in
\cite{GM-L} with the assistance of suitable software to deal with formal computations.

\subsection{Vector fields and $2$-dimensional Lie algebras}
In \cite{RR_separatrix}, codimension~$1$ foliations spanned by pairs of commuting vector fields were considered
and it was shown that this condition imposes strong constraints on the core dynamics of dicritical components of
the codimension~$1$ foliation in question. In particular, these constraints have proved to be strong enough to
yield the existence of separatrices for the foliation in question.

In view of the preceding, it was natural to wonder if the $1$-dimensional foliations arising from the vector fields
in question would have separatrices themselves. While the answer turned out to be affirmative, the assumption of
having two commuting vector fields can be weakened to encompass also the case of pairs of vector fields generating
the Lie algebra of the affine group. In fact, the following theorem was proved in~\cite{RR_secondjet}:

\begin{theorem}\cite{RR_secondjet}\label{teo_RR_secondjet}
Let $X$ and $Y$ be two holomorphic vector fields defined on a neighborhood $U$ of $(0,0,0) \in \C^3$ which are not linearly
dependent on all of $U$. Suppose that $X$ and $Y$ vanish at the origin and that one of the following conditions holds:
\begin{itemize}
  \item $[X,Y]=0$;
  \item $[X,Y]=c \, Y$, for a certain $c \in \C^{\ast}$.
\end{itemize}
Then there exists a germ of analytic curve $\mathcal{C} \subset \C^3$ passing through the origin and simultaneously
invariant under $X$ and $Y$.
\end{theorem}

The theorem above deserves a few additional comments.

\begin{itemize}
\item First, the fact that Theorem~\ref{teo_RR_secondjet} applies to pair of vector fields generating the Lie algebra
of the affine group is in stark contrast with the analogous problem for codimension~$1$ foliations. Indeed,
every homogeneous vector field of degree at least~$2$ together with the radial vector field
generate the Lie algebra of the affine group. In particular, Jouanolou's and similar examples of codimension~$1$ foliation
without separatrices arise from pairs of vector fields generating the Lie algebra of the affine group.

\item Whereas theorems asserting the existence of separatrices for foliations of dimension~$1$ holds interest
in their own right, they also have non-trivial applications in the general problem of understanding globally
defined holomorphic vector fields on compact complex manifolds, see Section~\ref{subsec:local_aspects}.
In particular,  the paper \cite{RR_secondjet} also includes a non-trivial applications of
Theorem~\ref{teo_RR_secondjet} in this direction.

\item Finally, a relatively minor but yet subtle issue that is worth pointing out is that
Theorem~\ref{teo_RR_secondjet} claims that $X$ and $Y$ possess a common invariant curve without {\it asserting
that the curve in question is invariant by the foliations associated with $X$ and $Y$}. To further clarify the
issue, it is enough to think of the $2$-dimensional vector field $x \partial /\partial x$: the axis $\{ x=0\}$ is
invariant by the vector field but does not constitute a separatrix for the associated foliation. In turn, it
might be asked if the foliations associated with
$X$ and $Y$ share an actual separatrix, possibly enlarging the notion of ``separatrix'' to include curves
fully constituted by singular points of the corresponding foliation. In particular, it is
easy to check that the existence of ``common separatrices'' always holds when $X$ is a homogeneous vector field
and $Y$ is the radial vector field. Indeed, in this case the
leaves of $\fol_Y$ are simply the radial lines. Concerning $\fol_X$,
since it is not a multiple of the Radial vector field, it induces a $1$-dimensional foliation on $\CP^2$ by means of the
one-point blow-up of $\C^3$ at the origin. The foliation in question possesses isolated singular points and it can easily be
checked that the radial line naturally associated with any of these singular points is invariant by $\fol_X$ as well. We
believe that the existence of a common separatrix for $\fol_X$ and $\fol_Y$ in the general case can also be established.
\end{itemize}

To finish the section, let us provide an outline of the proof of Theorem~\ref{teo_RR_secondjet}.

\begin{proof}[Sketch of Proof of Theorem~\ref{teo_RR_secondjet}]
Recall that the foliation associated with $X$ (resp. $Y$) is denoted by $\fol_X$ (resp. $\fol_Y$).
Let $\cald$ denote the codimension~$1$ foliation spanned by $X$ and $Y$. We have that ${\rm codim} \, ({\rm Sing} \,
(\cald)) \geq 2$. In other words, ${\rm Sing} \, (\cald)$ is of one of the following types: the union of a finite number of
irreducible curves, a single point (the origin), or simply empty (i.e. $\cald$ is regular). Since ${\rm Sing} \, (\cald)$ is
naturally invariant by $X$ and by $Y$, the result immediately holds if ${\rm dim} \, ({\rm Sing} \, (\cald))=1$. Hence we can
assume without loss of generality that ${\rm Sing} \, (\cald)$ has codimension at least~$3$. In other words,
either ${\rm Sing} \, (\cald)$ is reduced to the origin or it is, in fact, empty.

Since the singular set of $\cald$ has codimension at least~$3$, Malgrange Theorem \cite{malgrange} implies that
$\cald$ possesses a non-constant holomorphic
first integral $f$. Let then $S = f^{-1}(0)$ so that $S$ is an invariant surface for $\cald$, i.e. the irreducible
components of $S$ are separatrices for $\cald$. In particular, $S$ is invariant by both $X$ and $Y$. Next, note that
$S$ can be assumed to be irreducible. Otherwise, the intersection of any two irreducible components
of $S$ yields a curve invariant under both $X, \, Y$ and the conclusion holds.
The surface $S$ can then be assumed either regular or having an
isolated singularity at the origin (again if $S$ contains a curve of singular points this curve must be invariant
by $X$ and $Y$). At this point, a couple of remarks are in order:
\begin{itemize}
\item In the case where $S$ is smooth, each of the foliations $\fol_X$ and $\fol_Y$ possesses separatrices owing
to Camacho-Sad Theorem \cite{camachosad}. Still it remains to check that these foliations share a common separatrix.

\item As previously mentioned, in the case of singular surfaces, there are examples of foliations without separatrix
(cf.~\cite{C} or \cite{GM-L}). This phenomenon needs thus to be ruled out in the present case.
\end{itemize}

In general, we proceed as follows.
Consider the restrictions of $X$ and $Y$ to $S$ along with the corresponding tangency locus. This tangency locus
is not empty since both $X$ and $Y$ vanish at the origin.
Since the tangency locus ${\rm Tang} \, (X|_S, Y|_S)$ is invariant by both $X$ and $Y$, the
result immediately holds in the case where its dimension equals 1. So, we shall consider separately the case where
${\rm Tang} \, (X|_S, Y|_S) = \{ (0,0,0)\}$ and the case where ${\rm Tang} \, (X|_S, Y|_S) = S$.

Assuming that ${\rm Tang} \, (X|_S, Y|_S)$ is reduced to the origin. Then $S$ is a surface with
singular set of codimension at least 2 and
equipped with two vector fields that are linearly independent away from this an analytic set of codimension~$2$
or greater. This implies that tangent sheaf to $S$ is locally trivial
which, in turn, implies that $S$ is smooth since $S$ is a hypersurface in $\C^3$.
However, being smooth, $S$ is locally equivalent to $\C^2$ and the tangency locus of two vector
fields cannot be reduced to a single point. The resulting contradiction rules out this case.

Assume now that ${\rm Tang} \, (X|_S, Y|_S) = S$, i.e. the restrictions to $S$ of
$X$ and $Y$ coincide up to a multiplicative function (defined on $S$). The
existence of the desired common separatrix is then ensured in the case where $S$ is smooth by
Camacho-Sad theorem. It only remains to consider the case where $S$ has an isolated singular point at
the origin. The argument in this case relies on proving that the ($1$-dimensional) foliation induced on $S$ by either
$X$ or $Y$ possesses a non-constant holomorphic first integral. The level curve of this first
integral containing the origin then
yields the desired separatrix. Details can be found in~\cite{RR_secondjet}.
\end{proof}


\section{Semicomplete vector fields, automorphism groups, and separatrices}\label{semicompletevectorfields}

The object of this last section is a distinguished class of singularities of vector fields, namely the
{\it semicomplete (singularities of) vector fields}. Understanding this class of vector fields, both at global level and at
level of germs, is a problem
with interesting applications. As an example of application, we will see in Section~\ref{subsec:local_aspects}
that results on singularities of semicomplete vector fields yield insight in some
problems about bounds for the dimension of automorphism group of compact complex manifolds.
Another motivation to study these vector fields and their singular points stems from the very
fact that the {\it semicomplete property}\, is somehow akin to the Painlev\'e property for differential equations,
albeit the two notions are not equivalent. As a matter of fact, as it happens
with Painlev\'e property, semicomplete vector fields are also largely present - sometimes implicitly - in the literature of
Mathematical Physics.

The notion of {\it semicomplete singularity}\, was introduced in~\cite{Rebelo96}. The idea begins with the definition of
semicomplete vector fields on general open sets which is as follows.

\begin{defi}\cite{Rebelo96}\label{definition_semicompletevf}
A holomorphic vector field~$X$ defined on an open set $U$ of some complex manifold $M$ is said to be semicomplete (on $U$)
if for every~$p\in U$ there exists a connected domain~$V_p
\subset \C$, with~$0\in V_p$, and a map~$\phi_p : V_p \to U$ satisfying the following conditions:
\begin{itemize}
\item $\phi_p(0)=p$
\item $\phi'_p(T) = X(\phi_p(T))$, for every $T \in V_p$.
\item For every sequence~$\{T_i\}\subset V_p$ such that~$\lim_{i\rightarrow\infty} T_i = \hat{T} \in \partial V_p$ the sequence
$\{\phi_p(T_i)\}$ escapes from every compact subset of~$U$.
\end{itemize}
\end{defi}

The third condition in Definition~\ref{definition_semicompletevf} basically means that
$\phi_p : V_p \to U$ is a maximal solution of $X$ in a sense similar to the notion of ``maximal solutions'' commonly
used for real vector field and/or differential equations. In this sense, the definition is equivalent to saying that
a vector field is semicomplete if for every $p \in U$ the integral curve $\phi$ satisfying $\phi (0) =p$
has a {\it maximal domain of definition in $\C$}. Closely connected to the notion of
maximal domain of definition in $\C$, we can think of a local integral curve for a vector field $X$ and then
extending it over paths which is always possible as long as we stay in the domain of definition of $X$.
The vector field $X$ is then semicomplete if these extensions do not give rise to any monodromy and hence can
be merged together in a single (univalued) solution for $X$ which is naturally defined on a maximal domain
in $\C$.

Though global in essence, the above definition has also a {\it local character}\, that is singled out by
the following assertion: {\it if a vector field $X$ is semicomplete on $U$, then the restriction of $X$
to every subset $V$ of $U$ is semicomplete as well}. Thus the notion of {\it germ of semicomplete vector field},
and hence of {\it semicomplete singularity}, makes sense.
Furthermore, even at level of germs, the condition of being {\it semicomplete}\, is far from trivial and, in fact,
imposes strong constraints on the singular points of vector fields as pointed out in~\cite{Rebelo96}.
As a matter of fact, since its introduction, semicomplete singularities have proved time and again that they
capture almost all of the ``intrinsic nature'' of germs of vector fields admitting
actual global realizations as complete vector fields.

Germs of holomorphic semicomplete vector fields on $(\C^2,0)$ were classified by Ghys and Rebelo in the
papers~\cite{Rebelo96} and~\cite{Ghys-R}. In particular,
all these vector fields admit a non-constant holomorphic/meromorphic
first integral so that the dynamics associated with them is rather simple.

After this brief introduction to semicomplete vector fields, the remainder of the section will focus on two fundamental
questions related to them. The first question was somehow motivated by the results of Ghys and Rebelo in dimension~$2$ and asks the extent
to which the condition of being semicomplete may tame the {\it core dynamics}\, of
the corresponding foliation. In other words, we ask:

\vspace{0.1cm}

$\bullet$ Are there semicomplete vector fields exhibiting a genuinely complicated
core dynamics?

\vspace{0.1cm}

\noindent The second question was raised by E. Ghys long ago and, roughly speaking, involves deciding ``how degenerate''
can semicomplete singular points be. A prototypical question along these lines concerns semicomplete vector fields
with isolated singular points and can be formulated as follows:

\vspace{0.1cm}

$\bullet$ Is it true that the second jet of a semicomplete vector field at an isolated singular point is necessarily different
from zero?

\vspace{0.1cm}

\noindent This question is affirmatively answered in dimension~$2$ in the mentioned works by Ghys and Rebelo.
It remains open in higher dimension, though a number of partial results
are available in dimension~$3$.

Whereas the interest in ``taming'' the core dynamics associated to singularities of vector fields has already been emphasized, let us also point
out that the general question raised by E. Ghys has
applications to problems about bounds for the dimension of automorphism group of compact complex manifolds. This issue
will further be discussed
in Section~\ref{subsec:local_aspects}. For the time being, we will focus on the dynamics associated with
semicomplete singularities.


\subsection{Semicomplete vector fields with complicated dynamics - Guillot's work~\cite{guillot-IHES} }\label{subsec:sc_Halphen}

As previously mentioned, singularities of semicomplete vector fields have very simple dynamics in complex dimension~$2$.
In fact, even the global behavior of
semicomplete vector fields is amenable to detailed analysis, see \cite{Guillot-Rebelo}, \cite{guillotadvances}. However,
this is no longer the case in dimension~$3$ as follows from Guillot's deep work on
Halphen vector fields. This paragraph is basically devoted to summarizing the main dynamical issues appearing
in semicomplete Halphen vector fields while referring to~\cite{guillot-IHES} for a more comprehensive discussion.

Halphen vector fields were first considered by Halphen himself \cite{halphen-1}, \cite{halphen-2}. Apart from his contribution,
let us make clear that {\it all remaining results in this paragraph are due to Guillot and can be found in~\cite{guillot-IHES}}.
Up to linear equivalence, Halphen vector fields form a three parameters family of homogeneous polynomial vector fields
of degree~$2$ on $\C^3$ explicitly described as
\begin{align}\label{Halphenformula}
X = &\left[ \alpha_1 z_1^2 + (1-\alpha_1) (z_1 z_2 + z_1 z_3 - z_2 z_3) \right] \frac{\partial}{\partial z_1} + \\
&\left[ \alpha_2 z_2^2 + (1-\alpha_2) (z_1 z_2 - z_1 z_3 + z_2 z_3) \right] \frac{\partial}{\partial z_2} + \nonumber \\
&\left[ \alpha_3 z_3^2 + (1-\alpha_3) (-z_1 z_2 + z_1 z_3 + z_2 z_3) \right] \frac{\partial}{\partial z_3} \nonumber \\
\nonumber \end{align}
An alternate definition pointed out in~\cite{guillot-IHES} which already sheds some light in the intrinsic nature of these
vector fields is as follows.

\begin{defi}
A homogeneous polynomial vector field of degree~$2$ (a quadratic vector field for short) on $\C^3$ is
Halphen if it satisfies the following relation
\begin{equation}
[C,X] = 2R \, , \label{equation_definition-HalphenVF}
\end{equation}
where $C$ stands for a constant vector field and $R$ is the Radial vector field.
\end{defi}

The normal form indicated in~(\ref{Halphenformula}) is obtained as the solutions
of Equation~(\ref{equation_definition-HalphenVF}) for
$C = \partial /\partial z_1 + \partial /\partial z_2 + \partial /\partial z_3$.
Since both $C$ and $X$ are
homogeneous, Euler relations imply that we also have $[R,C] = -C$ and $[R,X] = X$. In turn, these
three relations together mean that the triplet
$\{R, \, C, \, X\}$ generates the Lie algebra of ${\rm SL} \, (2,\C)$.


Let $\fol_X$, $\fol_R$, and $\fol_C$ denote the $1$-dimensional foliations associated to the vector fields $X$, $R$ and $C$, respectively.
Once again, let $\widetilde{\C}^3$ denote the blow-up of $\C^3$ centered at the origin with projection
$\pi : \widetilde{\C}^3 \rightarrow \C^3$. The exceptional divisor $\pi^{-1} (0)$ is isomorphic to $\CP^2$ and the blow ups
of $\fol_X$, $\fol_R$, and $\fol_C$ will respectively be denoted by $\tilf_X$, $\tilf_R$, and $\tilf_C$.
Similarly, $\widetilde{X}$, $\widetilde{R}$, and $\widetilde{C}$ will stand for the blow ups of $X$, $R$, and $C$.
Next, recall that, whenever two vector fields commute, then the flow of one of them will preserve the foliation associated with the other.
This simple remark hints at a
basic property of Halphen vector fields. Indeed, since $X$ and $C$ commute up to the Radial vector field, the flow of $X$ ``tends''
to preserve the projection of the foliation arising from $C$ along the orbits of $R$. To make this remark accurate, we first note
that the space of orbits of $R$ is naturally identified with the exceptional divisor $\pi^{-1} (0) \simeq \CP^2$ though, on $\pi^{-1} (0)$,
$\widetilde{X}$ vanishes identically and $\widetilde{C}$ has poles. However, the restrictions $\tilf_{X \vert \pi^{-1} (0)}$
and $\tilf_{C \vert \pi^{-1} (0)}$ to $\pi^{-1} (0)$ of the foliations $\tilf_X$ and $\tilf_C$ have a specific property of
``mutual transversality'' which is reminiscent from the previous observation on commuting vector fields. This can be stated as follows:

\begin{defi}
Two (singular) foliations $\fol_1$ and $\fol_2$ are said to be mutually transverse if they are (regular and) transverse away from
an algebraic curve $C$ which, in addition, is invariant by both $\fol_1$ and $\fol_2$. In particular, the curve $C$ contains all
singular points of $\fol_1$ and of $\fol_2$.
\end{defi}

Keeping in mind that
$\tilf_{C \vert \pi^{-1} (0)}$ is nothing but a pencil of projective lines, the ``mutual transversality'' condition makes it
easy to work out the structure of $\tilf_{X \vert \pi^{-1} (0)}$ directly on $\pi^{-1} (0)
\simeq \CP^2$. Namely, we have:
\begin{itemize}
\item Generically, $\tilf_{X \vert \pi^{-1} (0)}$ leaves exactly $3$ projective lines $C_1$, $C_2$ and $C_3$ invariant.
These projective lines belong to the pencil $\tilf_{C \vert \pi^{-1} (0)}$ and they
intersect mutually at a radial singularity in $\pi^{-1} (0)$ (the base locus of the pencil) which is given in homogeneous
coordinates by $[1,1,1]$.
Also, the eigenvalues of $\tilf_{X \vert \pi^{-1} (0)}$ at $[1,1,1]$ are $1$ and $1$ (radial singularity).

\item In fact, $[1,1,1]$ is a radial singularity for the foliation in the $3$-dimensional space. In other words, the eigenvalue
of $\tilf_X$ at $[1,1,1]$ associated to the direction transverse to the exceptional divisor is again~$1$.

\item Away from the invariant projective lines $C_1$, $C_2$, and $C_3$, the foliation $\tilf_{X \vert \pi^{-1} (0)}$
is transverse to the remaining projective lines in the pencil $\tilf_{C \vert \pi^{-1} (0)}$.

\end{itemize}
Next, since $X$ is homogeneous, the dynamics of the foliation $\tilf_X$ on $\widetilde{\C}^3$ can basically be recovered
from the dynamics of the {\it core foliation}\, $\tilf_{X \vert \pi^{-1} (0)}$ on $\pi^{-1} (0) \simeq \CP^2$. We will return to this
point later.

%


In view of the preceding, let us first focus on the core foliation $\tilf_{X \vert \pi^{-1} (0)}$. Note that both
$\tilf_{X \vert \pi^{-1} (0)}$ and the pencil $\tilf_{C \vert \pi^{-1} (0)}$, the latter viewed as foliation, share
the singular point $[1,1,1] \in \pi^{-1} (0)$.
Consider the ($2$-dimensional) blow up of $\pi^{-1} (0) \simeq \CP^2$ at~$[1,1,1]$. The resulting surface
is the Hirzebruch surface $F_1$, the $\CP^1$-bundle over $\CP^1$ with a section of self-intersection~$-1$. On $F_1$,
the foliation (pencil) $\tilf_{C \vert \pi^{-1} (0)}$ becomes the standard fibration ${\rm P} : \, F_1 \rightarrow \CP^1$.
In turn, the transform $\fol_{X, F_1}$ of the foliation $\tilf_{X \vert \pi^{-1} (0)}$ on $F_1$ is regular on a neighborhood
of the $-1$-rational curve of $F_1$ (identified with the exceptional divisor $\pi^{-1} (0)$ of the blow up of $\CP^2$).
Also, there are $3$ fibers of ${\rm P}$ that are invariant by $\fol_{X, F_1}$ and these fibers will still be denoted by
$C_1$, $C_2$, and $C_3$ by evident reasons. Away from these three fibers, $\fol_{X, F_1}$ is regular and transverse to the
fibration induced by ${\rm P}$ on the open manifold $F_1 \setminus \{ C_1,C_2,C_3\}$.

The dynamics of $\tilf_{X \vert \pi^{-1} (0)}$ can naturally be read off the dynamics of $\fol_{X, F_1}$ which, in turn, is
essentially described by the {\it holonomy representation}. In fact, the restriction of $\fol_{X, F_1}$ to
the open surface $(F_1 \setminus \{ C_1,C_2,C_3\})$
is transverse to the restriction to $(F_1 \setminus \{ C_1,C_2,C_3\})$ of the fibration ${\rm P} : \, F_1 \rightarrow \CP^1$. Since the fibers of
${\rm P}$ are compact, Ehresmann's observation ensures that the restriction of ${\rm P}$ to the leaves of $\fol_{X, F_1}$
yields a covering map from the leaf in question to $\CP^1 \setminus \{z_1,z_2,z_3\}$, where $z_1, z_2, z_3$ are in natural
correspondence with the invariant fibers $C_1,C_2,C_3$. The dynamics of $\fol_{X, F_1}$ is therefore essentially encoded
in the holonomy representation, namely: the homomorphism $\rho$ from the fundamental group of $\CP^1 \setminus \{z_1,z_2,z_3\}$
to the group of automorphisms of the fiber of ${\rm P}$ arising from parallel transport along leaves of $\fol_{X, F_1}$.

%
%
%
%

Let $\pi_1 (\CP^1 \setminus \{z_1,z_2,z_3\})$ denote the fundamental group of $\CP^1 \setminus \{z_1,z_2,z_3\}$. Since
$\fol_{X, F_1}$ is holomorphic, the image of the holonomy representation $\rho$ is contained in the group of holomorphic diffeomorphisms
of $\CP^1$ which can be identified with $\psl$. The {\it holonomy group} $\Gamma$ of $\fol_{X, F_1}$ is the image of
$\pi_1 (\CP^1 \setminus \{z_1,z_2,z_3\})$ by $\rho$, i.e. $\Gamma \subset \psl$ is defined by
$\Gamma = \rho [\pi_1 (\CP^1 \setminus \{z_1,z_2,z_3\})]$.

Next, for each $i=1,2,3$, let $\xi_i \in \psl$ be the holonomy map obtained by lifting a small loop around $z_i \in \CP^1$
in the leaves of $\fol_{X, F_1}$. The M\"oebious transformations $\xi_1, \xi_2, \xi_3$ clearly generate the holonomy
group $\Gamma$ and satisfy the relation $\xi_1 \, \xi_2 \, \xi_3 = {\rm id}$. With the evident identifications,
the dynamics of $\Gamma$ on $\CP^1$ also accounts for the global dynamics of the foliation $\tilf_{X \vert \pi^{-1} (0)}$ on
$\CP^2$.

All of the preceding considerations apply to every Halphen vector field in the family defined by~(\ref{Halphenformula}), regardless
of whether or not they are semicomplete. To detect semicomplete Halphen vector fields in the family~(\ref{Halphenformula}),
we proceed as follows. First, notice that the singularities of $\tilf_X$ and of $\tilf_{X \vert \pi^{-1} (0)}$ do coincide.
Naturally there is the point $[1,1,1]$ lying at the intersection of all the lines in the pencil
$\tilf_{C \vert \pi^{-1} (0)}$. Moreover, around $[1,1,1]$, the foliation $\tilf_X$ is conjugate to the radial vector field in dimension~$3$.

To describe the structure of the remaining singular points, for
$i=1,2,3$, let $m_i = (\alpha_1 + \alpha_2 + \alpha_3 - 2)/\alpha_i$ provided that $\alpha_i \neq 0$,
and set $m_i = \infty$ otherwise. The remaining singular points of $\tilf_X$ are contained in the lines $C_1,C_2,C_3$ and
are as follows.
\begin{itemize}
	\item[(1)] If $m_i \neq \infty$. Then, aside from $[1,1,1]$, $\tilf_X$ possesses exactly two singular points $p_i$ and $q_i$
	in the line $C_i$. The eigenvalues of $\tilf_X$ at $p_i$ are $-1,-1,-m_i$ while at $q_i$ the eigenvalues are
	$-1,-1,m_i$. In both cases, the eigenvalues are ordered so that to the first eigenvalue corresponds to a direction transverse
	to the exceptional divisor, the second eigenvalue is associated with the direction of $C_i$ and the third eigenvalue is associated
	with a direction transverse to $C_i$ and contained in the exceptional divisor.

	\item[(2)] If $m_i = \infty$. Then, aside from $[1,1,1]$, $\tilf_X$ possesses a unique singular point $p_i=q_i$ in $C_i$.
	At this singular point, the eigenvalues are $-1,0,-1$ with the same ordering used in the above item.
\end{itemize}

When $m_i = \infty$, the holonomy map $\xi_i$ is a parabolic map in $\psl$ since $\tilf_{X \vert \pi^{-1} (0)}$ has a ($2$-dimensional)
saddle-node singularity at $p_i=q_i$ with strong invariant manifold transverse to $C_i$. Next, we have:

\begin{prop}\label{necessary_condition-Halphen}
Assume that $X$ is semicomplete and that $m_i \neq \infty$. Then $m_i$ is an integer (which can be assumed positive up to
reversing the roles of $p_i$ and $q_i$).
Moreover the holonomy map $\xi_i: \C\p(1) \to \C\p(1)$ is periodic of period $m_i$.
\end{prop}

\begin{proof}
Again, up to renaming $p_i$ and $q_i$, the singular point $q_i$ of $\tilf_X$ lies in the Siegel domain
and the eigenvalues of the mentioned foliation at the
singular point in question fulfill the conditions 1., 2., 3. and 4. of Theorem~$1$ in~\cite{Reis06} (or, equivalently, Theorem~$2.19$
in~\cite{lectureNotes_RR}). Furthermore, with the language of \cite{Reis06}, \cite{lectureNotes_RR},
the eigenvalue that can be ``separated'' from the others by
a straight line through $0 \in \C$ is the first eigenvalue (i.e. the eigenvalue associated with direction transverse to the
exceptional divisor). Consider then the separatrix
$S$ of $\tilf_X$ tangent to this direction. It is immediate to check that the restriction of $\widetilde{X}$ to $S$ is given,
in local coordinates, by $-z^2 \partial /\partial z$. Being $X$ semicomplete, there follows that the local holonomy map
of $\tilf$ arising from a small loop in $S$ encircling $q_i$ must agree with the identity, c.f. \cite{Ghys-R}.
Theorem~$1$ in~\cite{Reis06} then ensures that $\widetilde{\fol}_X$ is linearizable around $q_i$.

In particular, the foliation $\tilf_{X \vert \pi^{-1} (0)}$ is also linearizable around $q_i$. It follows that $\tilf_{X \vert \pi^{-1} (0)}$
possesses a separatrix transverse to $C_i$ and that the holonomy map arising from this separatrix is locally conjugate
to a rotation of angle $2\pi/m_i$. Because $\tilf_{X \vert \pi^{-1} (0)}$ is transverse to a fibration, this local holonomy map is, in fact,
the restriction of a global M\"oebius transformation $\xi_i \in \psl$ which, therefore, must verify $\xi_i^{m_i} = {\rm id}$.
\end{proof}

As an immediate consequence, we have the following

\begin{prop}\label{describingGamma_dimension2}
If $X$ is semicomplete, then the holonomy group $\Gamma \subset \psl$ describing the global dynamics of $\tilf_{X \vert \pi^{-1} (0)}$ is given by
\[
\Gamma = <\xi_1, \, \xi_2, \, \xi_3 : \, \, \xi_1^{m_1} = \xi_2^{m_2} = \xi_3^{m_3} = \xi_1 \xi_2 \xi_3 = {\rm id}> \, .
\]
In other words, $\Gamma$ is a triangular group.
\end{prop}

In Proposition~\ref{describingGamma_dimension2}, when $m_i = \infty$, the condition $\xi_i^{\infty} = {\rm id}$ must be interpreted
as simply saying that $\xi_i$ is parabolic. In the sequel, we shall also use the convention that $1/m_i =0$ provided that $m_i = \infty$.
In order to obtain semicomplete Halphen vector fields with complicate dynamics, we assume also that
\begin{equation}\label{hyperbolic_inequality}
m=\frac{1}{m_1} + \frac{1}{m_2} + \frac{1}{m_3} < 1 \, .
\end{equation}

The effect of inequality~(\ref{hyperbolic_inequality}) is just to rule out finitely many cases where the group $\Gamma$ is ``elementary'',
either finite or conjugate to a subgroup of the affine group of $\C$. Assuming $m_1,m_2,m_3$ fixed and as in~(\ref{hyperbolic_inequality}),
the resulting triangular group $\Gamma$ satisfy all of the following conditions:
\begin{itemize}
\item The group $\Gamma$ is unique (up to conjugation).
\item $\Gamma$ is discrete and non-elementary.
\item $\Gamma$ leaves a real projective line in $\CP^1$ invariant so that $\Gamma$ is actually a non-elementary Fuchsian group
(i.e. $\Gamma$ can also be viewed as a subgroup of ${\rm PSL} \, (2,\R)$).
\item The limit set $\Lambda (\Gamma)$ of $\Gamma$ coincides with the invariant circle $S^1$. In particular, $\Gamma$ acts properly
discontinuously on each connected component of $\CP^1 \setminus \Lambda (\Gamma)$.
\end{itemize}

As is well known, the dynamics of $\Gamma$ on its limit set $\Lambda (\Gamma) = S^1$ is very non-trivial: the dynamics has all orbits are dense
and it is ergodic with respect to the Lebesgue measure. Also stationary measures are unique (and hard to understand in detail). Clearly,
these issues are directly reflected in the saturated of $\Lambda (\Gamma)$ by the foliation $\tilf_{X \vert \pi^{-1} (0)}$ whose
dynamics is hence very non-trivial as well.

It is also convenient to say a few words on the actual dynamics of $\tilf_X$ on $\widetilde{\C}^3$ rather than limiting ourselves to
its core foliation $\tilf_{X \vert \pi^{-1} (0)}$. To describe this dynamics, we can follow essentially the same ideas used
to describe the foliation $\tilf_{X \vert \pi^{-1} (0)}$. Beginning with the pencil $\tilf_{C \vert \pi^{-1} (0)}$,
we define a family of surfaces in $\widetilde{\C}^3$ by considering the preimage of each line in $\tilf_{C \vert \pi^{-1} (0)}$
by the canonical projection $\Pi : \widetilde{\C}^3 \rightarrow \pi^{-1} (0) \simeq \CP^2$. More precisely, for every projective
line $D$ in the pencil $\tilf_{C \vert \pi^{-1} (0)}$, $\Pi^{-1} (D)$ is the line bundle over $\CP^1$ whose Chern class equals~$-1$.
Alternatively, by adding a ``section at infinity'', $\Pi^{-1} (D)$ can naturally be compactified into the Hirzebruch surface
$F_1$. In other words, up to adding a ``plane at infinity'' to $\widetilde{\C}^3$, we obtain a family of $F_1$ surfaces parameterized
by the lines in the pencil $\tilf_{C \vert \pi^{-1} (0)}$. Now, if we remove the three Hirzebruch surfaces sitting on the top of
the lines in $\tilf_{C \vert \pi^{-1} (0)}$ {\it that are invariant under $\tilf_{X \vert \pi^{-1} (0)}$}, it is straightforward
to conclude that $\tilf_X$ is {\it transverse to the fibration} by $F_1$-surfaces over $\CP^1 \setminus \{z_1,z_2,z_3\}$.
Thus, once again we obtain a representation $\overline{\rho}$ from the fundamental group of $\CP^1 \setminus \{z_1,z_2,z_3\}$ in the group
${\rm Aut}\, (F_1)$ of holomorphic diffeomorphisms of $F_1$. Let $\overline{\Gamma}$ be the image of $\overline{\rho}$, i.e.
the holonomy group of $\tilf_X$. Clearly, $\overline{\rho}$ is generated by the maps $\Xi_i$ obtained by lifting a small loop
around~$z_i$, $i=1,2,3$. The maps $\Xi_i$ can explicitly be computed. Fix a surface $F_1$ equipped with
coordinates $(x,w)$ where $x$ is projective coordinate on the projective line $\Pi (F_1)$ and $w$ is an affine coordinate on the fibers
of $F_1$ that equals zero in the intersection with the exceptional divisor. Then we have
\begin{equation}\label{definition_Xi-i}
\Xi_i (x,w) = (\xi_i (x) , \sqrt{\xi' (x)} \, w )
\end{equation}
c.f. \cite{RR-Iberomamericana}. Keeping in mind that the dynamics of $\tilf_X$ on $\widetilde{\C}^3$ and the dynamics of $\fol_X$
on $\C^3$ can be identified, what precedes can be summarized as follows (the slight abuse of language should not really lead to any
misunderstanding):

\begin{prop}\label{describingGamma_dimension3}
The dynamics of $\fol_X$ on $\C^3$ is essentially equivalent to the dynamics of the group $\overline{\Gamma} = \langle \Xi_1, \Xi_2,\Xi_3 \rangle$
on $F_1$. In particular, the $(-1)$-section of $F_1$ is invariant by $\overline{\Gamma}$ and the restriction of the action of
$\overline{\Gamma}$ to this section is nothing but the action of the triangular group $\Gamma$ on $\CP^1$.
\end{prop}

By now, we have provided a description of the (rather non-trivial) dynamics of Halphen vector fields such that the quantities $m_i$ are integers
satisfying the condition in~(\ref{hyperbolic_inequality}) and the reader is
referred to \cite{guillot-IHES} for additional information. However, strictly speaking, we still {\it do not know}\, whether or not
Halphen vector fields satisfying the conditions in question are, indeed, semicomplete. In fact, Proposition~\ref{necessary_condition-Halphen}
provides only necessary conditions for the vector field to be semicomplete. Hence, there remains the problem of checking that these conditions
are also {\it sufficient}.

Curiously enough the fact that the corresponding Halphen vector fields are semicomplete is basically included in
Halphen original papers \cite{halphen-1}, \cite{halphen-2}. Halphen begins his Note by pointing out that,
if $\phi$ is a solution of a Halphen vector field, then so is
\begin{equation}\label{Halphen_equation_projectivestructure}
\widetilde{\phi} = \frac{1}{(ct + d)^2} \, \phi \left( \frac{at + b}{ct + d} \right) - \frac{c}{ct + d} \, ,
\end{equation}
for every $a, \, b, \, c, \, d \in \C$ with $ad - bc \ne 0$. From this he concludes that all solutions can be described out of
a single ``known'' solution. He then goes on to obtain a particular solutions by skillfully manipulating theta functions. In this
sense, the converse to Proposition~\ref{necessary_condition-Halphen} can be derived from his work.

Yet, Guillot \cite{guillot-IHES} provides a different proof of the semicomplete nature of Halphen vector fields
satisfying the conditions in in Proposition~\ref{necessary_condition-Halphen}. Guillot's argument
dispenses with the remarkable identities satisfied by theta functions and, perhaps more importantly, lends itself well
to deep generalizations. We will close this paragraph by sketching this argument.

First, it is convenient to recall the basic notions of {\it translation, affine, and projective
structures}\, on Riemann surfaces since they play a role in the discussion below.

\begin{defi}
Let $S$ be a Riemann surface along with a covering $\{ (B_i, \varphi_i)\}$ by local coordinates. The covering $\{ (B_i, \varphi_i)\}$
is said to define a translation structure (resp. affine structure, projective structure) on $S$ if and only if the changes of
coordinates $\varphi_i \circ \varphi_j^{-1} : \varphi_j (B_i \cap B_j) \rightarrow \varphi_i (B_i \cap B_j)$ are restrictions
of translations of $\C$ (resp. affine maps, M\"oebius transformations).
\end{defi}

In particular, if $S$ is endowed with a nowhere zero holomorphic vector field $X$, then the covering whose local coordinates
are the inverse maps of the (local) solutions of $X$ endows $S$ with a {\it translation structure}. This simple remark will
be useful below.

Also, a translation structure (resp. affine structure, projective structure) gives rise to a {\it monodromy homomorphism} $\rho$
from the fundamental group of $S$ to the group of translations of $\C$ (resp. affine maps, M\"oebius transformations). Following
\cite{guillot-IHES}, \cite{Guillot-Rebelo}, denote by $S_{\rho}$ the covering space of $S$ associated with the kernel of $\rho$.
On $S_{\rho}$, we can define a {\it developing map}\, $\mathcal{D}_{\rho} : S_{\rho} \rightarrow \C$
(resp. $\C$, $\CP^1$). In fact, $S_{\rho}$ is the {\it smallest covering of $S$ on which a developing map is well defined}. This
developing map will be called {\it the monodromy-developing map} of the corresponding structure. Naturally, all developing maps
are well defined up to composition with an element of the corresponding group (translation, affine map, or M\"oebius transformations).

\begin{remark}\label{yetanotherdefinitionsemicompleteness}
The preceding offers us yet another equivalent way to define semicomplete vector fields on a Riemann surface, and thus in general
since a vector field will be semicomplete if and only if its restriction to each leaf of its associated foliation is semicomplete.
Namely, the vector field $X$ on the Riemann surface $S$ is semicomplete if and only if the monodromy-developing map of the
corresponding translation structure is {\it injective}. We are now ready to explain Guillot's argument.
\end{remark}

\vspace{0.1cm}

\noindent {\it Sketch of Guillot's proof that Halphen vector fields as in Proposition~\ref{necessary_condition-Halphen}
are semicomplete}. We might start by recalling that the vector fields $R$ and $C$ generate the Lie algebra of the affine group
${\rm Aff}\, (\C,0)$. Of course a similar remark applies to their blow ups $\widetilde{R}$ and $\widetilde{C}$.
Then we consider the Zariski open subset $W$ of $\widetilde{\C}^3$ given as the complement of $\pi^{-1} (0)$ and of the $3$ invariant
Hirzebruch surfaces. In the setting of Proposition~\ref{describingGamma_dimension3}, $W$ is a $U$-bundle over
$\CP^1 \setminus \{ z_1,z_2,z_3\}$, where $U \subset F_1$ is the Zariski open set defined as the
complement of the two rational sections of $F_1$. In particular, $U$ is in a natural correspondence with an orbit of ${\rm Aff}\, (\C,0)$.
Finally, $\tilf_X$ is transverse to the fibers of the fibration $W \rightarrow \CP^1 \setminus \{ z_1,z_2,z_3\}$ and admits
a global holonomy group determined by Proposition~\ref{describingGamma_dimension3}.

It suffices to show that the restriction of $\widetilde{X}$ to $W$ is semicomplete. Guillot basic observation is that $\widetilde{X}$
induces a natural projective structure on $\CP^1 \setminus \{ z_1,z_2,z_3\}$ viewed as the base of the $U$-bundle $W$. This deserves
a few comments. Small discs $B \subset \CP^1 \setminus \{ z_1,z_2,z_3\}$ can be identified with discs on the leaves $L$
of $\tilf_X$ by means of the fiber bundle structure. Next, each leaf $L$ of the restriction of $\tilf_X$ to $W$ is endowed with
a translation structure induced by $\widetilde{X}$. These transverse structures vary with the leaf but its {\it underlining projective
structure does not}. In fact, taking into account that a fiber $U$ of the $U$-bundle $W$ is identified with an orbit of
${\it Aff}\, (\C)$ - and thus parameterized by the flows of $R$ and of $C$, Equation~(\ref{Halphen_equation_projectivestructure})
can be interpreted as an identity involving the flows of $R$, $C$, and $X$ (or of their blow ups which amounts to the same).
With this interpretation, it becomes clear that the time taken by $\widetilde{X}$ to move between two fixed fibers $U_1$ and $U_2$
along leaves $L$ and $L'$ are related by a M\"oebius transformation. Thus the covering of the base $\CP^1 \setminus \{ z_1,z_2,z_3\}$
obtained by taking the inverses of the local solutions of $X$, as above, over all possible leaves of $\tilf_X$ defines
a projective structure on $\CP^1 \setminus \{ z_1,z_2,z_3\}$.

Next, consider the monodromy-developing map $\mathcal{D}_{\rho}$ for the projective structure $\CP^1 \setminus \{ z_1,z_2,z_3\}$. It is straightforward
to check that {\it representatives}\, for this developing-map can be obtained by simply considering the monodromy-developing
maps associated with the {\it translation structures induced by $\widetilde{X}$}\, on the leaves of $\tilf_X$ (or more accurately
of the restriction of $\tilf_X$ to $W$). In view of Remark~\ref{yetanotherdefinitionsemicompleteness}, there follows
that $X$ is semicomplete if and only if $\mathcal{D}_{\rho}$ is injective. Guillot's then ``compute'' the projective structure
in question by means of the Schwarzian operator so as to show that $\mathcal{D}_{\rho}$ is essentially
Schwarz triangular functions and the proof follows.\qed

\begin{remark}\label{notbeingquotedIguess}
In fairness, we should note that the material covered in this paragraph is essentially the first part of Guillot's paper
\cite{guillot-IHES}. The content of \cite{guillot-IHES} also includes realizing semicomplete Halphen vector fields as actual
{\it complete}\, vector fields on complex manifolds as well as several important applications to the study of ${\rm SL}\, (2\C)$
actions and homogeneous spaces.
\end{remark}

Let us close this paragraph with a couple of questions about dynamics of semicomplete vector fields, the first one being kind
of inevitable.

\vspace{0.1cm}

{\bf Problem 1}. Are there semicomplete vector fields with complicated dynamics which genuinely different from the dynamics
	obtained by means of Halphen vector fields?

\vspace{0.1cm}

Another interesting question which may or may not have a saying in the above problem concerns geodesic flows on semisimple Lie groups.
These geodesic flows have already been considered in works by S. Dumitrescu and by Elshafei-Ferreira-Reis, see \cite{sorin}, \cite{Ahmed}
and their references.
Given a (semisimple) Lie group $G$ and a left-invariant holomorphic metric on $\langle \, . \, \rangle$ on $G$, the complex geodesic flow
on $G$ can be expressed by a quadratic vector field defined on the Lie algebra of $G$ by means of the Euler-Arnold formalism. This yields
a particular, yet large and with geometric nature, class of quadratic vector fields. Referring to vector fields
in this class as {\it Euler-Arnold}\, vector fields, their dynamics is definitely worth study. Thus we can formulate
the following special case of the preceding question which, however, holds interest in its own:

\vspace{0.1cm}

{\bf Problem 2}. Are there semicomplete Euler-Arnold vector fields exhibiting complicated dynamical behavior?

\vspace{0.1cm}


\subsection{Local aspects of semicomplete vector fields and applications}\label{subsec:local_aspects}

Partly, the interest of semicomplete vector fields comes from the fact that they provide local obstructions
for a germ of vector field be realized as singularity of a complete one. In the sequel, we will talk
about {\it germs of semicomplete vector fields} or about {\it semicomplete singularities}\, as synonymous.

From the basic properties discussed at the beginning of this section, it follows that semicomplete vector fields can be viewed
as a ``local counterpart'' of complete ones.
In fact, a singularity that {\it is not semicomplete}\, cannot be realized by a complete vector field. In particular, it cannot be
realized by a globally defined holomorphic vector field on a compact manifold. The understanding of semicomplete singularities
is therefore useful to the description of holomorphic vector fields (globally) defined on compact manifolds.

To better explain this issue, it is convenient to center the discussion around a rather concrete and well known question
due to E. Ghys that can be formulated in terms of semicomplete vector fields as follows:
{\it let $X$ be a semicomplete holomorphic vector field on $(\C^n,0)$ with isolated singular points. Is it
true that $J^2 X (0) \ne 0$, i.e. must the second jet of $X$ at the singular point be different from zero}\,?

Ghys' original motivation seems to be related to problems about bounds for the dimension of automorphism group of compact complex
manifolds. To be more precise, consider a compact complex manifold $M$ and denote by ${\rm Aut} \, (M)$ the group of holomorphic
diffeomorphisms of $M$. It is well known that ${\rm Aut} \, (M)$ is a finite dimensional complex Lie group whose Lie algebra can
be identified with $\mathfrak{X} \, (M)$, the space of all holomorphic vector fields defined on $M$. A too na\"{\i}ve question,
would be to wonder if the dimension of ${\rm Aut} \, (M)$ can be bounded by a function of the dimension of $M$. It turns out,
however, that the dimension of the automorphism group of the Hirzebruch surface $F_n$ is $n + 5$ provided that $n \geq 1$.
In particular, already in the case of compact surfaces, the dimension of ${\rm Aut} \, (M)$ can be arbitrarily large. However,
analogous questions can be raised to better effect for specific classes of manifolds. For example, among projective manifolds
with Picard group isomorphic to $\Z$, Hwuang and Mok
asked if there is a $n$-dimensional manifold whose dimension of the automorphism group exceeds the dimension
of the automorphism group of $\CP^n$.

As a matter of fact, Ghys' question is part of a general principle with vaguely stated as follows: semicomplete singularities cannot be
``too degenerate''. Here it is convenient to explain how {\it limiting the extent to which a semicomplete singularity can be degenerate}\, becomes
a useful tool to deal with the previous questions. Consider a $n$-dimensional compact complex manifold $M$ and let
${\rm Aut} \, (M)$ and $\mathfrak{X} \, (M)$ be as above. Fix a point $p \in M$ and let $k \in \N$ be given. Finally,
let $\mathfrak{X}_p^k(M)$ stand for the set of holomorphic
vector fields with vanishing $k$-jet at $p$ and denote by $J^k_p(M)$ the space of $k$-jets at~$p$. The natural mappings
\[
\mathfrak{X}_p^k \, (M) \to \mathfrak{X} \, (M) \to J^k_p(M) \, ,
\]
give rise to a short exact sequence so that we have
\[
{\rm dim} \, \mathfrak{X} \, (M) \leq {\rm dim} \, \mathfrak{X}_p^k \, (M) + {\rm dim} \, J^k_p(M) \, .
\]
The dimensions of the jet spaces $J^k_p(M)$ are explicitly given in terms of~$k$ and
of $n = {\rm dim} \, (M)$. In particular, if for some $p \in M$ and $k \in \N$, we can obtain
bounds for ${\rm dim} \, \mathfrak{X}_p^k \, (M)$ in terms of ${\rm dim} \, (M)$ then bounds for
${\rm dim} \, ({\rm Aut} \, (M))$ follow immediately.
For example, suppose that we happen to know that for a certain class of compact manifolds every singularity of a globally defined
holomorphic vector field is necessarily isolated.
Then, assuming Ghys conjecture holds, it follows that ${\rm dim} \, \mathfrak{X}_p^3 \, (M) = 0$ and therefore
the dimension of ${\rm Aut} \, (M)$ would be bounded by $(n^3 + 3n^2 +2n)/2$. Of course, in general, non-isolated singularities also
appear so that it is convenient to be able to handle them as well.

Aside from introducing the notion of semicomplete singularity, the content
of~\cite{Rebelo96} can fairly be summarized by the following theorem:

\begin{theorem}\cite{Rebelo96}\label{teo:scdimension2}
Let $X$ be a holomorphic semicomplete vector field on $(\C^2,0)$. If the origin is an isolated singular point for $X$, then $J^2_0 X \ne 0$.
\end{theorem}

The proof of Theorem~\ref{teo:scdimension2} relies on Camacho-Sad theorem on the existence of separatrices for foliations
on $(\C^2,0)$. Indeed, since the singular set of $X$ is reduced to the origin,
the restriction of $X$ to any analytic invariant curve going through the origin cannot vanish identically.
Furthermore, this restriction is still
a semicomplete vector field. Considering then the restriction of $X$ to a separatrix, whose
existence is ensured by Camacho-Sad theorem, the problem becomes essentially reduced to the one-dimensional
situation (whether or not the separatrix is smooth). The resulting (one-dimensional) problem is settled in the same paper by direct methods.

%

The question on whether or not Ghys conjecture holds for semicomplete vector fields in higher dimensions
is hence natural.
The first deep investigations involving semicomplete vector fields in
higher dimensions were conducted by A. Guillot
in~\cite{guillotFourier}, and \cite{guillot-IHES} (here ``higher'' means~$\geq 3$). The mentioned papers by
Guillot contain, in particular, numerous examples of quadratic semicomplete vector fields
exhibiting a wide range of geometric and dynamical behaviors.
Among these examples, we have already discussed the case of Halphen vector fields that have complicated dynamics and no (non-trivial)
holomorphic/meromorphic first integral (cf. Section~\ref{subsec:sc_Halphen}). Moreover, Guillot's work also make clear
that in dimensions~$\geq 3$, an exhaustive classification of all semicomplete vector fields with zero linear part - paralleling
the list provided in~\cite{Ghys-R} - is unlikely to exist or, at least, it would be too long to be truly useful.

This is therefore a good moment to elaborate on the difficulties in
extending to $(\C^3,0)$ the general classification results in dimension~$2$ of~\cite{Rebelo96},
\cite{Ghys-R}, not to mention the more general results of \cite{Guillot-Rebelo} encompassing also meromorphic vector fields.
Indeed, whether or not obtaining these generalizations is a tall order, it certainly seems useful to explicitly
list some of the new difficulties arising in dimensions greater than~$2$. Aside from the existence of
{\it core dynamics}, that is a general difficulty already emphasized in this work, the following issues are worth mentioning.

\smallbreak

\begin{itemize}

\item[1.] The basic approach to Ghys conjecture stemming from~\cite{Rebelo96} consists of finding a separatrix. Namely,
the following holds: let $X$ be a semicomplete (holomorphic) vector field on
$(\C^n,0)$ with an isolated singularity at the origin. If $X$ possesses a separatrix, then $J^2_0 X \ne 0$. However,
as previously seen, Gomez-Mont and Luengo \cite{GM-L} have proved that separatrices do not exist in general for germs of $1$-dimensional foliations
on $(\C^n,0)$, $n\geq 3$ (cf. Section~\ref{Allsortsofseparatrices}).

\smallbreak

\item[2.] The examples provided in \cite{GM-L}, however, are not semicomplete so that it is conceivable that all semicomplete vector field
possesses a separatrix. While this seems to suggest that Ghys conjecture may be proved by showing that semicomplete vector fields
do have separatrices, a direct approach to the latter question does not seem feasible.

\smallbreak

\item[3.] A more promising point of view regarding item~2 above consists of noticing that the detailed classification of semicomplete
vector fields in dimension~$2$, as developed in~\cite{Ghys-R} or in~\cite{Guillot-Rebelo}, dispenses with Camacho-Sad theorem. In fact,
these deeper analysis yield directly the classification. Hence the existence of separatrices for {\it semicomplete vector fields}\, on
$(\C^2,0)$ becomes a {\it corollary}, as opposed to a statement needed a priori.

\smallbreak

\item[4.] In dimension~$2$, a fundamental ingredient permeating virtually all works on singularities of vector fields or foliations is
the resolution theorem of Seidenberg \cite{seidenberg}. Since resolutions theorems for $1$-dimensional foliations have been established
in the past few years, c.f. Section~\ref{resolution_blowups}, this initial difficulty has now been overcome. In fact, as far
as semicomplete vector fields are concerned, a totally faithful analogue of Seidenberg's theorem is available in dimension~$3$
as will be seen below.

\smallbreak

\item[5.] Difficulties, however, are not limited to reduction of singularities procedures nor to the phenomenon of core dynamics.
For example, assume our objective is to establish Ghys conjecture by means of proving the existence of separatrices (in which case
the role played by core dynamics is significantly reduced, c.f. Section~\ref{separatrices_1dimension}). Assume, in addition, that we are
given a holomorphic foliation admitting a simple reduction of singularities. In dimension~$3$, the existence of saddle-node singularities
appearing in the resolution procedure cannot easily be ruled out. In particular, codimension~$2$ saddle-nodes
(i.e. with two eigenvalues equal to zero) may appear and these singularities are still poorly understood.

\end{itemize}

The remainder of this paper is to complement the above list with further comments and results, some of them proposing simpler
approaches that can be effective pending on the specific application targeted.

Concerning items~1 and~2, some partial results have been proved in \cite{RR_secondjet}. In fact, recall that Theorem~\ref{teo_RR_secondjet}
states that in the case we are given two holomorphic vector fields $X$ and $Y$ yielding a representation of a Lie algebra of dimension~2
and not everywhere parallel, then they possess a common separatrix. By elaborating on this theorem, the following weaker version
of Ghys conjecture in dimension~$3$ was proven in \cite{RR_secondjet}:

%

\begin{theorem}\cite{RR_secondjet}\label{teo_RR_nonvanishing_secondjet}
Consider a compact complex manifold $M$ of dimension~$3$ and assume that the dimension of ${\rm Aut}\, (M)$ is at least~$2$. Let $Z$ be an
element of $\mathfrak{X}\, (M)$ and suppose that $p \in M$ is an isolated singularity of~$Z$. Then
$$
J^2 (Z) \, (p) \neq 0 \, ,
$$
i.e. the second jet of $Z$ at the point $p$ is different from zero.
\end{theorem}

The reader will note that, as far as estimates on the dimension of automorphism groups are concerned, Theorem~\ref{teo_RR_nonvanishing_secondjet}
is as effective as an affirmative solution to Ghys conjecture in dimension~$3$ in the sense that if the additional assumption needed for
Theorem~\ref{teo_RR_nonvanishing_secondjet} is not verified, then the dimension of ${\rm Aut}\, (M)$ is at most~$1$.

%
%

One of the advantages of looking for bounds for the dimension of ${\rm Aut}\, (M)$ by means of local considerations is that the
results obtained are essentially valid for open manifolds as well. For example, studying finite dimensional Lie group actions on
Stein manifolds is an active topic in several complex variables whose roots lie in a classical work of
Suzuki~\cite{suzuki}. In this direction, our techniques yield:

\begin{theorem}\cite{RR_secondjet}\label{firstjet_forStein}
Let $N$ denote a Stein manifold of dimension~$3$ and consider a finite dimensional
Lie algebra $\mathfrak{G}$ embedded in $\mathfrak{X}_{\rm comp} (N)$ (the space of complete holomorphic vector fields on $N$).
Assume that the dimension of $\mathfrak{G}$ is at least~$2$. If $Z$ is an element of $\mathfrak{G} \subseteq
\mathfrak{X}\, (M)$ possessing an isolated singular point $p \in N$, then the linear part of $Z$ at $p$ cannot vanish,
i.e. $p$ is a non-degenerate singularity of~$Z$.
\end{theorem}

We may point out that the automorphism group of a Stein manifold is not a finite dimensional Lie group in general. Indeed, even $\C^2$
has an infinite dimensional group of automorphisms with hardly any non-trivial structure of Lie group.
This difficulty is avoided in the statement of Theorem~\ref{firstjet_forStein}
by the assumption that, from the beginning, we are dealing with some finite dimensional Lie algebra: owing to Lie theorem,
such Lie algebra can always be integrated
to yield a (complete) action of the corresponding Lie group. This part of the statement actually holds for arbitrary complex manifolds
of dimension~$3$ (whether or not they are compact or Stein) and parallels Theorem~\ref{teo_RR_nonvanishing_secondjet} in the sense that
the second jet of a vector field $Z \in \mathfrak{G}$ will never be zero at an isolated singular point. The role played by the Stein
condition is to ensure that the {\it first jet}, rather than the second one, is different from zero at isolated
singular points.

To close this paper, let us go back to resolution theorems as discussed in Section~\ref{resolution_blowups} and further
sharpen the results under the {\it additional assumption of having semicomplete vector fields}. As already mentioned, resolution theorems
always play a central role in singularity theory and sharp resolution statements exist for $1$-dimensional foliations
in dimension~$3$. Yet, given their importance, it is convenient to have available the simplest possible resolution statements in every
circumstance. In particular, it is natural to wonder if semicomplete singularities or other special classes of singular points
allow for simpler resolution theorems facilitating a more detailed analysis of their structures.

The possibility of having simpler resolution theorems valid for semicomplete singularities was also considered in~\cite{survey_RR}
whose initial motivation was, in fact, to obtain a resolution
theorem for {\it semicomplete singularities} that would faithfully parallel Seidenberg theorem for foliations on $(\C^2,0)$.
This type of statement is useful to approach problems such as Ghys conjecture or to investigate
compact complex manifolds of dimension~$3$ equipped with holomorphic vector fields. In this setting, Theorem~\ref{teo:B}
below is proved in~\cite{survey_RR}.

\begin{theorem}\cite{survey_RR}\label{teo:B}
Let $X$ be a semicomplete vector field defined on a neighborhood of the origin of $\C^3$ and denote by $\fol$ the holomorphic
foliation associated with $X$. Then one of the following holds:
\begin{enumerate}
  \item The linear part of $X$ at the origin is nilpotent (non-zero).

  \item There exists a finite sequence of (standard) blow-ups along with transformed foliations
$$
\fol = \fol_0 \stackrel{\pi_1}\longleftarrow \fol_1 \stackrel{\pi_2}\longleftarrow \cdots
\stackrel{\pi_r}\longleftarrow \fol_r
$$
such that all of the singular points of $\fol_r$ are elementary. Moreover, each blow-up map $\pi_i$ is centered in the
singular set of the corresponding foliation $\fol_{i-1}$. In other words, the foliation $\fol$ can be resolved by means of
standard blow-ups.
\end{enumerate}
\end{theorem}

Let us emphasize that item~1 in Theorem~\ref{teo:B} involves the linear part of the vector field X rather than the linear part of
the associated foliation $\fol$. In fact, it is the linear part of $X$ that has to be (nilpotent) {\it non-zero from the outset}.
Furthermore, this property is ``universal'' in the sense that it does not depend on any sequence of blow-ups/blow-downs carried out.
In particular, we can choose a ``minimal model'' for our manifold and the corresponding transform of $X$ will still have non-zero
nilpotent linear part at the corresponding point. From what precedes, the following also deserves to be singled out.

\begin{corol}\cite{survey_RR}\label{coro:C}
Let $X$ be a semicomplete vector field defined on a neighborhood of $(0,0,0) \in \C^3$
and assume that the linear part of $X$ at the origin is equal to zero. Then item~(2) of Theorem~\ref{teo:B} holds.
\end{corol}

Accurate normal forms for persistent nilpotent singular points were provided in the same paper, c.f. Theorem~\ref{Microlocalversion_TheoremA}.
However, not all of them need to be semicomplete or, indeed, realized as singularity of a complete flow.
Taking into account the global setting of complete vector fields, it
is natural to wonder if there is, indeed, complete vector fields inducing a foliation with singular points that cannot be resolved
by standard blow ups as in item~(2) of Theorem~\ref{teo:B}. As a matter of fact, these singularities {\it do exist}\, and an
explicit example is provided by
the polynomial vector field
\[
Z = x^2 \frac{\partial}{\partial x} + xz \frac{\partial}{\partial y} + (y - xz)\frac{\partial}{\partial z} \, .
\]
Although $Z$ is not complete on $\C^3$, it
can be extended to a complete vector field defined on a suitable {\it open manifold}. In particular, the point corresponding to
the origin of the above coordinates $(x,y,z)$ constitutes a nilpotent
singular point of $Z$ that cannot be resolved by means of standard blow-ups with centers in the singular set.

Finally, we might emphasize that the example above involves a complete vector field defined on an open manifold. We might then
ask if this phenomenon still occurs in the far more restrictive context of compact manifolds of dimension~$3$.
Since in the compact case the completeness condition becomes automatic, we are simply asking whether or
not there is a compact complex manifold of dimension~$3$ equipped with a (global) holomorphic vector field $X$ which exhibits
a singular point that cannot be resolved by means of standard blow ups. This time, the answer turns out to be negative
as the following holds:

\begin{corol}\label{thelastcorollary_compact}
Let $\fol$ be the foliation associated with a vector field $X$ globally defined on some compact complex manifold $M$ of dimension~$3$. Then every
singular point of $\fol$ can be resolved by a sequence of standard blow ups.
\end{corol}

In closing, let us just point out that both Corollary~\ref{coro:C} and Corollary~\ref{thelastcorollary_compact} are strictly
speaking by-products of the methods used to prove Theorem~\ref{teo:B} rather than formal consequences of the statement
of this theorem.


\end{document}